\documentclass[11pt,a4paper]{article}

\usepackage{amsmath}
\usepackage{amssymb}
\usepackage{amsfonts}
\usepackage{amsthm}
\usepackage[T1]{fontenc}
\usepackage[utf8]{inputenc}

\newcommand{\func}[1]{\operatorname{#1}}

\newtheorem{theorem}{Theorem}[section]
\newtheorem{proposition}[theorem]{Proposition}
\newtheorem{lemma}[theorem]{Lemma}
\newtheorem{corollary}[theorem]{Corollary}
\newtheorem{definition}[theorem]{Definition}

\begin{document}

\author{Javier P\'{e}rez \'{A}lvarez \\
Dpto. Matem\'{a}ticas Fundamentales UNED\\
Juan del Rosal 10 28040 Madrid\\
jperez@mat.uned.es}
\date{}
\title{Some geometric perspectives on Contact Hamiltonian Dynamics}
\maketitle

\begin{abstract}
This article presents a unified overview of contact Hamiltonian geometry as
a natural framework for the description of dissipative and non-conservative
systems. Starting from the symplectic cover of a contact manifold, we
clarify the structural relation between contact and symplectic dynamics and
show how dissipation is geometrically encoded through the contact structure
and the Reeb vector field. Following the introduction, which provides a
guided overview of the subject through key references, a dedicated section
illustrates the scope of the theory through applications ranging from
thermodynamics, statistical mechanics, and integrable and KAM systems to
field theories, quantum and Lie systems, optimal control, control theory,
and economic models, where dissipation, constraints, and optimization play a
central role. The subsequent sections review and adapt classical
constructions of geometric mechanics, such as integrability,
Hamilton--Jacobi theory, symmetries, and reduction, to the contact setting.
Particular emphasis is placed on recent developments in contact reduction,
Dirac structures, and constrained systems. The article also surveys emerging
approaches to the geometric quantization of contact manifolds and discusses
how ideas from generalized geometry provide a unifying perspective for
symplectic, contact, and related frameworks.

\end{abstract}

\section{Introduction via key references}

Contact geometry is much more than the odd-dimensional sibling of symplectic
geometry. Soon after its emergence, it became evident that contact geometry
exhibits a high degree of topological flexibility, particularly with the
advent of contact surgery. See Eliashberg \cite{E} for foundational results
on the classification of overtwisted contact structures in dimension three
(characterized by embedded overtwisted disks with Legendrian boundary, where
the contact plane field exhibits flexible behavior), and Weinstein \cite{W}
for the development of contact surgery techniques within symplectic
handlebodies. By contrast, the well-known rigidity of symplectic structures
presents a geometric shield difficult to deform; see, e.g., Gompf \cite{G}.

The foundations of contact geometry can already be found in classical
treatises, such as Godbillon work \cite{God}, which provides an organized
and accessible account of fundamental geometric structures of great
significance in mechanics. The monograph by Libermann and Marle \cite{LM} is
a cornerstone reference that provides a systematic exposition of symplectic
and contact geometry and their deep connections with analytical mechanics,
with great clarity of exposition. For a modern, concise, and accessible
presentation of the fundamental concepts and techniques of contact geometry,
together with its elder sister, symplectic geometry, we refer to the book by
Eslami Rad \cite{ES}, which also highlights their applications. In addition,
Banyaga and Houenou \cite{BH} provide a clear and comprehensive introduction
to symplectic and contact manifolds. One should also mention the seminal
work by Blair \cite{B}, which investigates the intricate interplay between
Riemannian geometry and contact structures, covering a wide range of topics
such as contact structures on tangent sphere bundles, curvature properties
of contact metric manifolds, Boothby--Wang fibrations, as well as complex
contact and complex contact metric manifolds, among others.

Recent advances in the topology of contact geometry have revealed deep
connections with both geometric and differential topology. In this sense,
the seminal work of Geiges \cite{Ge} develops contact topology from three to
higher dimensions, with a particular emphasis on contact surgery and other
fundamental constructions of contact manifolds. For the symplectic
background, we direct the reader to McDuff and Salamon \cite{MS}, which
includes discussions of transverse contact structures where relevant.

Symplectic geometry not only provides the natural geometric framework for
Hamiltonian dynamics, but also plays a fundamental role in the transition to
quantum theory through a variety of geometric and algebraic quantization
schemes. Accordingly, we highlight key references in this area, which
exemplify the continuing effort to geometrize one of the fundamental engines
of modern science. In this context, building on the idea that dynamics
shapes geometrical structures, Cari\~{n}ena \textit{et al.} \cite{CIGG}
provide a detailed exposition of the interplay between geometric structures
and dynamics in both classical and quantum systems. Their work covers
essential aspects of classical dynamics, including Jacobi, Poisson,
symplectic, Hamiltonian, and Lagrangian structures, as well as the geometric
framework underlying quantum dynamics, particularly the geometry of
Hermitian spaces, the K\"{a}hler structure of the space of pure quantum
states, and geometric descriptions of quantum evolutions and state spaces.
Another key reference is de Gosson \cite{dG}, which, intended for readers
with a solid background in quantum theory, explores symplectic geometric
techniques employed in quantization. Key topics addressed include the Maslov
index, the Heisenberg group, Weyl calculus, the metaplectic representation,
Wigner functions, and the symplectic treatment of Lagrangian submanifolds of 
$\mathbb{R}^{2n}$ within the Weyl--Wigner formalism, where quantum objects
are associated to classical phase-space structures.

However, this extensive framework primarily describes isolated systems and
does not account for interactions with an environment, which give rise to
dissipative dynamics and irreversibility. Developing models for dissipative
systems is therefore essential for a comprehensive understanding of physical
phenomena. The challenge lies in the application of the Lagrangian and
Hamiltonian formalisms to processes involving energy loss or transfer due to
friction, drag, heat dissipation, or other environmental effects. A natural
approach is to represent dissipative mechanical systems within the framework
of contact Hamiltonian systems, which encode energy dissipation at the
geometric level. In this context, contact geometry has emerged as a powerful
framework for the study of dissipative mechanical systems, with a
significance expected to rival that of symplectic geometry in classical
mechanics. Seminal contributions in this area can be found in \cite%
{Brr,Br,BLN,BT,LS,GGMRRx}.

Building on this premise, contact geometry provides an appropriate formalism
for incorporating dissipation into mechanical systems, offering a rich
geometric structure that yields deep insight into their dynamics. Its
fundamental elements, such as the contact form and the associated Reeb
vector field, provide essential tools for studying non-conservative forces
and energy-dissipative processes. A solid foundation for these aspects can
be found in \cite{BCT,dLL}. For a comprehensive treatment of these
foundational ideas, see also \cite{LL,Pe,SLLD}.

Perhaps surprisingly, contact Hamiltonian dynamics, as well as conserved and
dissipative quantities, admit their own variational principles, which are
crucial for understanding the properties and dynamics of these systems
beyond the conservative setting. Relevant references on this topic can be
found in \cite{LTW,VBS}.

\section{Applications and theoretical advancements in Contact Hamiltonian
Dynamics}

This section highlights the broad applications of contact geometric dynamics
in modeling dissipative and non-conservative systems. Its influence extends
across traditional domains such as thermodynamics, integrable systems,
quantum mechanics, field theories, Lie systems, statistical mechanics, and
optimal control, and even into interdisciplinary fields including economics
and neuroscience. We provide a concise overview of these contributions,
drawing attention to representative references. In the subsequent sections,
the discussion turns to theoretical advances, where recent developments have
extended classical constructions, such as integrability, Hamilton--Jacobi
theory, reduction procedures, Dirac structures, Geometric Quantization, and
Hitchin Generalized Geometry, into the contact framework.

Rajeev \cite{Raj} argues that the theory of differential forms, and in
particular contact geometry, naturally arises in classical physics, in part
as a response to the need for a geometric understanding of thermodynamics.
This is exemplified by the fundamental relation between infinitesimal
changes in the internal energy $U,$ the volume $V$ and the entropy $S$ of a
material subjected to pressure $P$ and temperature $T,$%
\begin{equation}
dU+PdV-TdS=0.  \tag{1}
\end{equation}%
This perspective has, in turn, enabled the extension of contact geometric
principles beyond the realm of traditional mechanics.\medskip

\noindent \textbullet\ One area that has benefited significantly from the
application of contact geometry, unsurprisingly given its thermodynamic
origins is \textit{Thermodynamics}, particularly in the study of
nonequilibrium energy conversion processes. The fundamental relation (1)
marks one of the earliest appearances of contact Hamiltonian dynamics in
physical applications. More precisely, this condition defines a Legendre
submanifold $L$ of a contact manifold $(\mathcal{T},\eta )$, on which a
suitable contact Hamiltonian $H$, vanishing along $L$, can be introduced.
This construction induces a natural flow that preserves the equilibrium
manifold $L$ through geometric Hamilton--Jacobi theory. Within this
geometric setting, contact geometry provides a powerful framework for
analyzing the interplay between energy conservation and entropy production
in thermodynamic systems and their transformations.

For a more comprehensive perspective, the literature provides a broad
spectrum of relevant contributions. Among them, the following may be
regarded as key starting points: Bravetti, L\'{o}pez-Monsalvo, and Nettel 
\cite{BLN}, who introduce contact Hamiltonian dynamics as a natural
framework for modeling thermodynamic processes and describe the relationship
between the Hamiltonian function and irreversible entropy production; and
Grmela \cite{Gr}, who employs contact geometry in the context of mesoscopic
thermodynamics, systems in which thermal fluctuations are relevant but which
lie beyond the scope of classical thermodynamics, thereby opening new
theoretical directions.

Additional key contributions include the work of Goto \cite{Go}, who
develops the use of Legendre submanifolds in contact manifolds as a
geometric representation of equilibrium thermodynamics and shows how
transitions from nonequilibrium states toward equilibrium can be described
by a class of contact Hamiltonian vector fields; and Mrugala \cite{Mr}, who
analyzes the approach to thermodynamic equilibrium through contact
transformations. Further developments along these lines can be found in the
works of Mrugala \cite{Mr2} and Eberard, Maschke, and van der Schaft \cite%
{EMS}, which build upon and extend the Hamiltonian formalism, offering new
perspectives on the study of complex and irreversible systems. A
comprehensive overview is provided by Bravetti \cite{Br}, who reviews the
geometric formulation of equilibrium thermodynamics within the framework of
contact geometry, and explores the associated metric structures arising from
thermodynamic fluctuation theory.

Finally, Rajeev \cite{Raj} suggests that contact geometry may establish
novel connections between thermodynamics and quantum theory. In particular,
he develops a quantization scheme for contact manifolds that may lead toward
a quantum theory of thermodynamics, thereby bridging classical and quantum
descriptions of thermodynamic systems. In this light, Einstein's celebrated
tribute to thermodynamics, emphasizing the simplicity and universality of
its principles and quoted by Rajeev at the beginning of his paper, resonates
strongly amid the emergence of these geometric structures.\medskip

\noindent \textbullet\ The theory of \textit{Integrable Systems}, understood
as the study of dynamical systems possessing a rich algebraic and geometric
structure that allows for exact solvability, plays a central role in the
analysis of nonlinear dynamics in classical and quantum physics. A
comprehensive introduction to this subject is given in the monograph by
Babelon, Bernard, and Talon \cite{BBT}, which covers a broad spectrum of
topics, including Calogero--Moser models, the KP hierarchy, the KdV and
Gelfand--Dickey hierarchies, Toda field theories, the sine--Gordon model,
the Riemann--Hilbert problem, tau- and theta-functions, and Painlev\'{e}
equations.

Within the framework of geometric mechanics, the deep relationship between
integrability and underlying geometric structures is presented in the now
classical work of Perelomov \cite{Perelomov}, which emphasizes the central
role of Lie algebras in classical integrable systems, symplectic geometry,
and Hamiltonian reduction. Further developments along these lines, with
particular relevance to symplectic geometry, and with implications for later
developments in contact geometry, are presented by Reyman and
Semenov-Tian-Shansky \cite{Reyman-Semenov}. Their work develops a
group-theoretical approach to integrability, emphasizing the interplay
between Lie groups, symplectic reduction, and the theory of classical $R-$%
operators, which provide a systematic framework for the construction of
integrable systems via Hamiltonian reduction. Both approaches reveal how
integrability emerges from symmetry and reduction principles encoded in
Lie-theoretic and symplectic structures.

Building upon the deep interplay between integrability and geometric
structures, recent advances have extended the role of contact geometry in
the construction of higher-dimensional integrable systems. A seminal
contribution in this direction is the work of Sergyeyev \cite{Se}, which
introduces a novel class of $(3+1)-$dimensional integrable systems. By
employing contact vector fields and algebraic Lax pairs, Sergyeyev develops
a systematic method for associating integrable systems with pairs of
rational functions, thereby extending the scope of integrability theory
beyond the traditional symplectic and Poisson frameworks.

The geometric approach to the complete integrability of a Hamiltonian system
with $n$ degrees of freedom and Hamiltonian $H(q^{i},p_{j})$ defined on a
symplectic manifold $M$ is classically characterized by the
Arnold--Liouville theorem. This theorem assumes the existence of $n$
independent integrals of motion $F_{1}=H,$ $F_{2}(q,p),...,F_{n}(q,p)$ in
involution, $\left[ {F_{i},F_{j}}\right] =0$, and guarantees the existence
of action--angle coordinates, providing a detailed description of the system
dynamics and long-term behavior. In this case, $M$ admits a Lagrangian
foliation, and the solutions of the Hamiltonian dynamics evolve on the
leaves of this foliation.

Introduced by A. S. Mishchenko and A. T. Fomenko \cite{MF}, the concept of
non-commutative integrability extends the classical Liouville--Arnold notion
to a broader framework. A Hamiltonian system on a symplectic manifold $M$ of
dimension $2n$ is said to be non-commutatively integrable if there exists a
finite-dimensional Lie algebra F of integrals of motion that is non-abelian
under the Poisson bracket and satisfies $\mathrm{dim}F+\mathrm{rank}F=2n$.
This generalization provides a richer setting for the study of integrable
systems, encompassing situations with non-abelian symmetries or singular
invariant foliations.

Within this context, Jovanovic \cite{J} investigates non-commutative
integrability in the framework of contact geometry, introducing methods for
the construction of action--angle variables and the analysis of the
corresponding dynamics. See also \cite{Vi} and \cite{JJ}, where partial
versions of the Arnold--Liouville theorem are established for Hamiltonian
systems restricted to hypersurfaces of contact type. In this setting,
integrability is not required on the entire phase space; instead, the
invariant hypersurface is foliated by invariant tori arising from the
contact integrable structure.\medskip

\noindent \textbullet\ The \textit{Kolmogorov--Arnold--Moser (KAM) Theory},
a cornerstone of Hamiltonian dynamics, asserts that under certain conditions
small perturbations of an integrable Hamiltonian system preserve most
quasi-periodic solutions, thereby ensuring the overall stability of the
system. Building on classical variational principles, the article by Wang,
Wang, and Yan \cite{WWJ} introduces an implicit variational principle for
contact Hamiltonian systems, extending the Herglotz variational principle to
this setting. Within this framework, the authors adapt and extend key ideas
from weak KAM theory, originally developed for symplectic systems, which
focuses on viscosity solutions of the Hamilton--Jacobi equation and captures
invariant structures even in the absence of smooth solutions. This work
characterizes action-minimizing trajectories in dissipative systems and
provides a rigorous mathematical framework to describe integrability in
non-conservative settings. Despite this variational antecedent, a KAM theory
in the contact setting is still in its early stages of development.
Nevertheless, several closely related developments already exist. For
instance, KAM theory has been extended to non-conservative systems in \cite%
{CCdLl}, while a KAM theory for conformally symplectic systems has been
developed in \cite{CdLl}.\medskip

\noindent \textbullet\ \textit{Field theories}, traditionally formulated
within the framework of classical Hamiltonian and Lagrangian mechanics, do
not intrinsically account for dissipative effects. Addressing this
limitation, the article by Gaset \textit{et al.} \cite{GGMR} introduces a
contact-geometric framework to describe the geometric structure of
dissipative field theories. The authors define the notions of $k-$contact
structures and $k-$contact Hamiltonian systems, which generalize both
contact Hamiltonian systems in mechanics and $k-$symplectic Hamiltonian
systems in field theory. This approach allows for a natural incorporation of
dissipation into the equations of motion, balance laws, and symmetries,
providing a powerful tool for the description of field theories in which
energy is not conserved. Building on this geometric foundation, the same
authors further develop in \cite{GGMRb} a $k-$contact Lagrangian formulation
for nonconservative field theories. This formulation extends the variational
principles of contact mechanics to the field-theoretic setting, yielding $k-$%
contact Euler--Lagrange equations that generalize the classical equations by
incorporating dissipative terms. The authors also analyze symmetries and
dissipation laws in this context, showing how these symmetries are
associated with quantities that evolve due to dissipation, in close analogy
with conserved quantities in conservative systems. Together, these works
provide a unified geometric language for the formulation and analysis of
dissipative field theories. \medskip

\noindent \textbullet\ As already introduced in Section 1, the work by de
Gosson \cite{dG} provides a comprehensive perspective on the connections
between symplectic geometry and \textit{Quantum Mechanics}, emphasizing how
geometric structures can underpin the formulation of quantum theory and
highlighting the historical and conceptual links between classical
Hamiltonian dynamics and quantum systems. It provides an excellent example
of how geometric structures can be used to describe physical phenomena
across different scales.

The study of dissipation in quantum dynamics involves more complex
interactions than in classical mechanics, requiring advanced geometric tools
to model its effects accurately. The work by Ciaglia, Cruz, and Marmo \cite%
{CCM} demonstrates how contact geometry can be applied to describe
dissipative phenomena at the quantum level, providing a geometric framework
for the evolution of open systems described by the
Gorini--Kossakowski--Lindblad--Sudarshan (GKLS) equation. The authors show
how the GKLS dynamical generator can be decomposed into Hamiltonian and
dissipative components, using contact geometry to obtain a unified
description of reversible and irreversible dynamics. This approach expands
the scope of contact geometry by offering a geometric understanding of the
evolution of open and dissipative quantum systems.

\medskip

\noindent \textbullet\ The study of \textit{Lie Systems} provides a bridge
between classical Lie theory and the analysis of nonlinear differential
equations, offering a geometric framework for understanding systems whose
solutions can be reconstructed from a finite set of particular solutions.
Introduced as a generalization of linear superposition principles, Lie
systems arise when the dynamics of a first-order differential equation are
governed by a time-dependent vector field taking values in a
finite-dimensional Lie algebra of vector fields. This structure applies
powerful algebraic techniques to solve such equations and shows deep
connections with symmetries, integrability, and reduction methods in both
classical and quantum mechanics. Recent advances, as outlined by Cari\~{n}%
ena and de Lucas \cite{CL}, have expanded the theory to include stratified
Lie systems and $k-$symplectic generalizations, thereby extending its range
of applications to areas such as control theory, thermodynamics, and the
geometric formulation of open quantum systems. Lie systems provide a
framework for addressing complex dynamical problems in settings where
traditional methods prove insufficient.

Within this framework, concrete applications of $k-$symplectic geometric
methods to the analysis of superposition rules and general properties of Lie
systems can be found in \cite{LV}. In this work, de Lucas and Vilari\~{n}o
develop novel $k-$symplectic geometric techniques to investigate
superposition rules, time-independent constants of motion, and other
fundamental properties of these systems. Their approach extends the scope of 
$k-$symplectic geometry and demonstrates its effectiveness in capturing the
intricate dynamics of systems for which traditional symplectic methods prove
insufficient. The authors provide a powerful geometric framework for
understanding the behavior of complex dynamical systems, ranging from
classical mechanics to field-theoretic models.

In a further extension of Lie systems, contact dynamics are incorporated
through the notion of contact Lie systems, as introduced in \cite{LR}. In
this work, de Lucas and Rivas define and analyze contact Lie systems as
systems of first-order differential equations whose integral curves are
generated by time-dependent vector fields taking values in a
finite-dimensional Lie algebra of Hamiltonian vector fields associated with
a contact structure. This framework generalizes the concept of Lie--Hamilton
systems to the contact setting, enabling the systematic study of dissipative
and non-conservative dynamical systems that lie beyond the scope of
traditional symplectic or Hamiltonian methods. The authors investigate
fundamental properties such as Liouville-type theorems, contact
Marsden--Weinstein reductions, and Gromov non-squeezing theorems, formulated
specifically for contact Lie systems. Moreover, they illustrate the
applicability of these systems in a variety of physical contexts, including
thermodynamics, control theory, and open quantum systems, where dissipation
and non-conservative interactions play a central role. This work highlights
the versatility of contact Lie systems as a powerful framework for modeling
complex dynamical behaviors in both classical and modern physics.\medskip

\noindent \textbullet\ \textit{Statistical Mechanics} employs probabilistic
methods to analyze systems with a large number of particles, bridging
classical mechanics and thermodynamics. Macroscopic behavior is described in
terms of microscopic properties through the use of statistical ensembles,
such as the canonical ensemble (constant temperature) and the microcanonical
ensemble (constant energy). In their seminal work \cite{BT}, Bravetti and
Tapias extend these concepts to non-conservative systems by employing the
framework of contact geometry. They derive generalized Hamilton equations
that incorporate dissipative terms and establish a contact version of
Liouville Theorem, ensuring the preservation of a generalized phase-space
volume under the contact flow. Furthermore, the authors construct invariant
microcanonical and canonical measures adapted to this setting, demonstrating
how contact geometry naturally accommodates the statistical description of
physical processes in which energy is not conserved.

In a subsequent work \cite{BT2}, the same authors introduce a deterministic
algorithm known as \textit{contact density dynamics}, designed to generate
arbitrary prescribed target distributions in phase space. This algorithm
operates within a non-Hamiltonian framework on an extended phase space,
where a friction term explicitly depending on the contact density is
incorporated into the dynamics. Contact density dynamics provides a
consistent approach that enables the simulation of complex systems in
regimes where conventional thermostatting methods fail, including
energy-dissipating and nonequilibrium systems.

Taken together, this contact-geometric approach extends the scope of
statistical mechanics to dissipative processes, such as open thermodynamic
systems, and provides a geometric framework for understanding the behavior
of complex systems both in and out of equilibrium.\medskip

\noindent \textbullet\ \textit{Optimal Control} theory seeks to determine
the best strategy to influence the evolution of a dynamical system by
minimizing or maximizing a prescribed objective functional. It relies on the
Pontryagin Maximum Principle, which plays a central role in applications
ranging from engineering and economics to biology and beyond. This principle
introduces a Hamiltonian function defined in terms of the system dynamics
and a cost function, and establishes that an optimal control must maximize
this Hamiltonian at each instant of time. Introduced in the 1950s \cite{PBGM}%
, its lasting relevance arises from the fact that maximizing the Hamiltonian
is significantly more tractable than directly solving the original
infinite-dimensional control problem. In \cite{BM}, a geometric formulation
of this principle is developed, providing an intrinsic framework for optimal
control problems and showing deep connections between classical mechanics
and optimal control theory. Building on this foundation, the work in \cite%
{LLMoptimalcontrol} investigates the interplay between optimal control,
contact dynamics, and the Herglotz variational problem, developing a unified
contact-geometric framework to address optimal control in systems subject to
non-conservative forces and constraints. For a concise and insightful
perspective on how contact geometry enriches the understanding of optimal
control systems, see also \cite{O}.\medskip

\noindent \textbullet\ \textit{Control Theory} addresses the design of
systems exhibiting dynamical behavior, with the main goal of ensuring that a
system follows a desired trajectory or satisfies prescribed criteria such as
stability, accuracy, and robustness. A central concept is feedback, whereby
the system output is continuously measured and used to adjust the input in
real time. Within the framework of Hamiltonian systems, this perspective
naturally leads to the study of input--output systems, which play a
fundamental role in applications from robotics and engineering to biological
systems, where controlling the behavior of complex physical systems is
essential.

A clear introduction to this topic is provided in \cite{RMS}, where
input--output Hamiltonian systems defined on contact manifolds are studied.
The authors analyze the preservation of the contact structure in a
closed-loop feedback, in which the system output is fed back into the
dynamics. They characterize the resulting closed-loop contact forms in terms
of their associated Reeb vector fields, showing that the contact structure
is preserved under feedback. A central result of this work is that
asymptotic stabilization to an equilibrium point cannot be achieved through
this mechanism. Instead, stabilization is possible only on certain invariant
Legendre submanifolds. This finding has important implications for the
design of feedback control systems, particularly in engineering applications
where the preservation of underlying geometric structures is essential for
ensuring stability.

Models describing interactions with the environment give rise to a wide
range of applications. In particular, extending port-Hamiltonian
formulations to incorporate stochastic dynamics opens new ways for modeling
complex physical processes subject to random perturbations. A notable
example is presented in \cite{CDM}, where Cordoni \textit{et al.} extend the
theory of port-Hamiltonian systems to include stochastic effects arising
from environmental interactions or measurement noise. The resulting
framework has extensive applications, ranging from molecular dynamics to the
rapidly advancing field of robotic control.

\medskip

\noindent \textbullet\ The applicability of the structures of analytical
mechanics to \textit{Economics} is illustrated in the work of Russell \cite%
{R}. The central idea is that economic data on prices and quantities arise
from actions that optimize objective functions on a symplectic manifold,
where quantities are given by derivatives of profit or utility with respect
to prices, reflecting a necessary condition for profit maximization. Within
this framework, the work shows how the Samuelson area condition, a
fundamental concept in economic theory introduced by Paul Samuelson, admits
a natural reinterpretation in symplectic terms. This perspective provides a
rigorous geometric foundation for understanding economic equilibrium and
optimization, in which the interplay between prices and quantities mirrors
the underlying geometric duality of symplectic manifolds.

In \cite{Sw}, symplectic structures are employed to reinterpret foundational
concepts in economics. In his work Foundations of Economic Geometry,
Swierstra explores how symplectic geometry can be used to model economic
processes, such as those described by Say's Law and the Phillips Curve, by
interpreting income as a $1-$form that generates the underlying symplectic
structure. This approach provides a geometric framework for understanding
economic equilibrium and optimization that goes beyond the static models of
neoclassical economics. Although Swierstra work focuses on symplectic
geometry, it naturally opens the door to further generalizations based on
contact geometry, which may be more effective in describing irreversible
processes and dynamic constraints inherent in economic systems.\medskip

\noindent \textbullet\ Even more remarkable is the connection between the
mathematical structures considered in this work and the fields of \textit{%
Neuroscience and Cognitive Science}, as explored in \cite{P}. In his book
Elements of Neurogeometry, Petitot investigates how visual neurons process
optical stimuli and how the functional architecture of the brain implements
geometric algorithms to construct spatial perception. In particular, the
book develops mathematical models of the primary visual cortex in which the
highly specific organization of neural connectivity is described using
sub-Riemannian geometry and contact structures, especially in relation to
the processing of visual contours and spatial representations. From this
perspective, Petitot argues that the cerebral cortex operates as a dynamical
system in which geometric structures emerge from collective neural
interactions, reflecting the underlying neurogeometric organization of
perception.

Along these lines, for readers interested in the role of mathematical
structures in the study of complex brain processes, a complementary
reference is \cite{GC}. In this work, Gabbiani and Cox provide a
comprehensive introduction to the mathematical foundations of neuroscience,
covering essential tools such as differential equations, linear algebra,
probability theory, stochastic processes, and statistical analysis. This
text complements the geometric approaches discussed above by supplying the
quantitative foundations required for modeling neural dynamics, signal
processing, and decision-making mechanisms in the brain.

\section{Our fundamental geometric arena}

Next, we focus on the geometric framework underlying contact dynamics. Our
starting point will be the symplectification of contact Hamiltonian vector
fields, a powerful technique that involves introducing an additional
dimension to transform the contact structure into a symplectic one. This
process also reveals deep connections between their respective dynamical
behaviors. This geometric approach offers new perspectives both in its
theoretical structure and in its applications to theoretical physics and
engineering. It has been rigorously established by foundational works such
as those of Arnold in \cite{A} and \cite{AG}, as well as by Libermann and
Marle in \cite{LM}.

Also, the influence of symplectification in contact geometry has inspired
the development of alternative geometric frameworks that have proved
fruitful in the study of contact dynamics. A notable example is the work of
Grabowska and Grabowski \cite{GG}, where first jet bundles of line bundles
provide a natural geometric framework for the definition of contact
Hamiltonians and Lagrangians. Within this setting, a contact-adapted
Hamilton--Jacobi theory is formulated in terms of sections of line bundles
rather than scalar functions. This approach yields a powerful geometric
viewpoint on dissipative systems and suggests a broad potential for
applications beyond the realm of classical mechanics.

A foundational approach to differential graded contact geometry and Jacobi
structures was developed by Mehta \cite{Me}. Subsequently, Grabowski \cite%
{Gra} provided a systematic and conceptually unified treatment of contact
and Jacobi structures on graded supermanifolds, where contact structures are
interpreted as symplectic principal $\mathrm{GL}(1,\mathbb{R})-$bundles.
Within this framework, the author investigates contact Courant algebroids,
viewed as contact counterparts of classical Courant algebroids, and
characterizes Kirillov (or Jacobi) brackets as homological Hamiltonians on
linear contact manifolds.

By considering a symplectification of a contact manifold, Khesin and
Tabachnikov \cite{KT} extend the notion of integrability to contact
dynamics. In contrast with the symplectic case, where integrability is
characterized by the existence of an invariant Lagrangian foliation, the
contact setting requires a flag of two foliations adapted to the contact
structure: a Legendrian foliation together with a co-Legendrian foliation.
This characterization highlights how contact geometry enriches the classical
theory of integrable systems, providing new tools to investigate dynamical
settings in which dissipative and conservative features naturally coexist.

As a final example, the work in \cite{FQ} proposes a general geometric
method for constructing contact algorithms, understood as numerical
integrators compatible with the underlying contact structure of dynamical
systems. The approach relies on the correspondence between contact geometry
on $\mathbb{R}^{2n+1}$ and conic symplectic geometry on $\mathbb{R}^{2n+2}$,
which allows such algorithms to be systematically derived from symplectic
methods originally developed for Hamiltonian systems. In this way, the
framework establishes a direct link between abstract contact geometry and
concrete numerical simulations.

We now turn to the precise mathematical formulation of these structures. We
briefly recall standard definitions and properties needed in the sequel. A
symplectic manifold is a $2n-$dimensional manifold $M^{s}$ endowed with a
closed $2-$form $\omega _{2}$ such that%
\begin{equation*}
\ker \omega _{2}=\{X\in \mathfrak{X}(M_{s}):i_{X}\omega _{2}=0\}=\{0\}.
\end{equation*}%
A submanifold $N\subset M_{s}$ is called coisotropic if for every point $%
p\in N$ the $\omega _{2}-$orthogonal complement $\left( T_{p}N\right) ^{\bot
}$ is included in $T_{p}N.$ It is clear that in this case\ $2n-\dim N\leq
\dim N$ and hence $n\leq \dim N.$ In the same way, $N$ is called isotropic
if for every point $p\in N$ the tangent space $T_{p}N$ is included in $%
\left( T_{p}N\right) ^{\bot }.$ It holds that $\dim N\leq n.$ When
\thinspace $N$ is both isotropic and coisotropic it is called Lagrangian
submanifold and for every point $p\in N$ we have $T_{p}N=\left(
T_{p}N\right) ^{\bot },$ consequently $\dim N=n.$

The Hamiltonian vector field associated to a smooth \ function $f$ on $M^{s}$
is the vector field $X_{f}$ defined by the condition $i_{X_{f}}\omega
_{2}=df.$ The Poisson bracket of two smooth functions $f$ and $g$ on $M$ is
defined by $[f,g])=\omega _{2}(X_{f},X_{g})$ and endows $C^{\infty }(M)$
with a Lie algebra structure. It is easy to see that if the submanifold $N$
is locally defined in a neigborhood $U$ by the zero set of the functions $%
f_{1},...,f_{k},$ then the Hamiltonian vector fields $X_{f_{i}}$ constitute
a basis for $\left( T_{p}N\right) ^{\bot },$ $(p\in U\cap N).$ It is a
remarkable fact the the coisotropic condition on $N$ is equivalent to the
stability of its sheaf of ideals on $M^{s}$ under the Poisson bracket (see,
for example \cite{dLR}). This equivalence highlights the deep connection
between geometric properties of submanifolds and algebraic structures in
Poisson geometry.

Before introducing the formal definition of a contact structure that is most
convenient to our purposes, we first present an intuitive, geometrical, and
classical description. Let $M_{c}$ be an odd-dimensional manifold and let $%
\eta $ be a nowhere-vanishing $1$-form on $M_{c}$. At each point $q\in M_{c}$%
, the kernel of the $2$-form $d_{q}\eta $ has dimension at least one. We are
interested in the situation where this kernel does not intersect the
hyperplane $\ker \eta _{q}$, that is, 
\begin{equation*}
\ker (d_{q}\eta )\cap \ker (\eta _{q})=\{0\}.
\end{equation*}%
This condition implies that 
\begin{equation*}
\dim \ker (d_{q}\eta )=1.
\end{equation*}%
When it holds at every point $q\in M_{c}$, the hyperplane distribution $\ker
\eta $ is said to be non-degenerate and defines a contact distribution on $%
M_{c}$.

Moreover, this condition excludes the possibility of expressing the $2$-form 
$d\eta $ as a wedge product of the form 
\begin{equation*}
d\eta =\eta \wedge \theta ,
\end{equation*}%
for some $1$-form $\theta $ on $M_{c}$. By the Frobenius theorem, this
implies that the contact distribution $\ker \eta $ is maximally
non-integrable. In fact, as will follow from our subsequent developments,
the highest-dimensional integral distribution associated with $\ker \eta $
has dimension $n+1.$

We now introduce a line-bundle formulation of a contact structure that is
not tied to the existence of a global $1-$form, while still pursuing the
geometric goal of defining a contact distribution on the tangent bundle.
This perspective, already anticipated in the approaches discussed above, as
in \cite{GG}, constitutes the germ of modern topological viewpoints of broad
scope.

\begin{definition}
Let $M_{c}$ be a manifold of dimension $2n+1$ and let $\mathcal{C}$ be a
sheaf $1-$forms on $M_{c}$ such that every point of $M_{c}$ has a
neighborhood $U$ where there exist local canonical coordinates $%
z,x_{1},...,x_{n},y_{1},...,y_{n}$ for which the $1-$form%
\begin{equation*}
\alpha _{U}=dz-\sum_{i}y_{i}dx_{i}
\end{equation*}%
locally generates $\mathcal{C}$ on $U.$ A manifold $M_{c}$ equipped with
such a sheaf $\mathcal{C}$ is called a contact manifold, and $\alpha _{U}$
is referred to as a local contact form.
\end{definition}

For each point $p\in M,$ denote by $\mathcal{C}_{p}$ the vector space formed
by the values at $p$ of the local sections of $C$ in neighborhoods of $p.$
Thus, $\mathcal{C}_{p}$ is a line in $T_{p}^{\ast }M$ not containing the
origin.

\begin{definition}
A diffeomorphism $f:(M_{c},\mathcal{C})\rightarrow (M_{c},\mathcal{C})$ is a
contactomorphism if $f^{\ast }\mathcal{C}=\mathcal{C}.$ A vector field $X\in 
\mathfrak{X}(M_{c})$ is an infinitesimal contactomorphism if its
corresponding $1-$parameter group of transformations $\phi _{t}$ is formed
by contactomorphisms, 
\begin{equation*}
\phi _{t}:(M_{c},\mathcal{C})\rightarrow (M_{c},\mathcal{C}),\text{ }\forall
t\in \mathbb{R}.
\end{equation*}%
It is clear that%
\begin{equation*}
L_{X}\mathcal{C}\subset \mathcal{C}.
\end{equation*}
\end{definition}

\begin{definition}
Given the contact structure $(M_{c},\mathcal{C})$ we shall denote by $M^{s}$
the set%
\begin{equation*}
M^{s}=\left\{ \left( p,\beta _{p}\right) ,\text{ }p\in M,\text{ }\beta
_{p}\in \mathcal{C}_{p}\right\}
\end{equation*}%
endowed with the projection 
\begin{equation*}
\pi :M^{s}\rightarrow M_{c}:\left( p,\beta _{p}\right) \longmapsto p
\end{equation*}%
and the differentiable structure defined as follows. Given an open set $U$
with canonical coordinates $z,x_{1},...,x_{n},y_{1},...,y_{n}$ we define
coordinates $t,z,x_{1},...,x_{n},$ $y_{1},...,y_{n}$ on $\pi ^{-1}(U)$ as
follows: $t\left( p,\beta _{p}\right) \in \mathbb{R}^{\ast }$ is such that%
\begin{equation*}
t\left( p,\beta _{p}\right) ^{-1}\cdot \beta
_{p}=d_{p}z-\sum_{i}y_{i}(p)d_{p}x_{i}.
\end{equation*}
\end{definition}

In this manner, we define the canonical $1-$form on $M^{s}$ by%
\begin{equation*}
\omega _{\left( p,\beta _{p}\right) }(X)=\beta _{p}(\pi _{\ast }X),\text{ }%
X\in T_{\left( p,\beta _{p}\right) }(M^{s}).
\end{equation*}%
Then, $\omega $ can be locally expressed by%
\begin{equation}
\omega =tdz-\sum_{i}ty_{i}dx_{i}  \tag{2}
\end{equation}%
where $(t,z,x_{i},y_{j})$ are the local coordinates in the domain $\pi
^{-1}(U),$ with $U\subset M_{c}$ being a chart equipped with the canonical
coordinates $(z,x_{i},y_{j})$. With a slight abuse of notation, using the
previously introduced notation, we will write the $1-$form $\omega $ as 
\begin{equation}
\omega =t\alpha _{U}  \tag{3}
\end{equation}%
where $\alpha _{U}$ is the local contact form from Definition 1.

From the expression (2) one easily deduce that $\Omega =d\omega $ is a
symplectic form on $M^{s}.$

\begin{definition}
The pair $(M^{s},\Omega =d\omega )$ shall be called the symplectic cover of
the contact structure $(M_{c},\mathcal{C}).$ If there exists a global
section of the sheaf $\mathcal{C}$, then the contact structure is determined
by a globally defined form $\eta $ on $M_{c}$, which is called the contact
form. In this case, the contact structure $C$ is called trivial or
cooriented.
\end{definition}

\noindent \textbf{Remark.} Alternative constructions of homogeneous
symplectic manifolds associated with contact structures appear in the
literature, primarily for extended tangent bundles $T^{\ast }Q\times \mathbb{%
R}.$ For instance, \cite{vSM}, presents a construction on the projective
bundle $\mathcal{P}(T^{\ast }(Q\times \mathbb{R})),$ while \cite{ibanez} is
based on the symplectic form $d(e^{-t}\eta )$ defined on the contact
canonical structure $(T^{\ast }Q\times \mathbb{R},\eta ).$ Both
constructions are symplectomorphic, although the underlying geometric
constructions, as well as the very notion of homogeneity (as discussed
below), are fundamentally different. As already anticipated in the
introductory discussion of this section,\ Grabowska and Grabowski introduce
in \cite{GG} a novel geometric approach to contact Hamiltonian mechanics
valid for general not necessarily cooriented contact structures. Unlike
traditional formulations, which rely on a global contact form $\eta ,$
contact Hamiltonians in this framework are no longer functions on the
contact manifold $M$ but rather $1-$homogeneous functions on a $GL(1,\mathbb{%
R})-$principal bundle $P\rightarrow (M_{c},C).$ This perspective, based on
the differential graded framework introduced by Mehta \cite{Me}, also
underlies the further developments by Grabowski in \cite{Gra}, where deep
connections between contact geometry and graded and supergeometry, as well
as Courant and Jacobi structures, are established. Our construction provides
an explicit realization of this principal bundle, mirroring the canonical
Liouville form in cotangent bundle geometry, and constituting the geometric
foundation for our analysis. In particular, this explicit realization will
also provide the basis for our subsequent calculations.\medskip

In order to illustrate the construction, we consider the special case in
which $M_{c}$ is connected and cooriented by a global contact form $\eta .$
In this case, the symplectic cover $M^{s}$ decomposes into the union of two
disjoint, open, connected components,%
\begin{equation*}
M^{s,+}=\{(p,t\eta _{p}),\quad p\in M_{c},t>0\},\text{ }M^{s,-}=\{(p,t\eta
_{p}),\quad p\in M_{c},t<0\}.
\end{equation*}

A is well known, the Liouville vector field of the exact symplectic
structure on $M^{s}$ is the unique vector field $\mathbb{L}$ on $%
(M^{s},\Omega =d\omega )$ such that%
\begin{equation*}
i_{\mathbb{L}}\Omega =\omega .
\end{equation*}%
In the local coordinates $(t,z,x_{i},y_{j})$ of the chart $\pi
^{-1}(U)\subset M_{c},$ we have%
\begin{equation*}
\mathbb{L}\mathcal{=}t\frac{\partial }{\partial t}.
\end{equation*}

A vector field $T$ on $(M^{s},\Omega )$ will be called $k-$homogeneous, if $%
\mathcal{L}_{\mathbb{L}}T=kT$ ($k\in \mathbb{Z}).$ In particular, the
symplectic structure $\Omega $ on $M^{s}$ is $1-$homogeneous. Indeed, we
have 
\begin{equation*}
\mathcal{L}_{\mathbb{L}}\omega =\omega .
\end{equation*}

For a smooth function $f\in C^{\infty }(M^{s}),$ we consider the Hamiltonian
vector field $X_{f}$ in the standard way, defined by the equation $%
i_{X_{f}}\Omega =df.$ By taking the Lie derivative $\mathcal{L}_{\mathbb{L}}$
and using the relation $[\mathcal{L}_{\mathbb{L}},i_{X_{f}}]=i_{[\mathbb{L}%
,X_{f}]}$ we arrive at the equality of the vector fields on $M^{s},$%
\begin{equation}
\lbrack \mathbb{L},X_{f}]=X_{\mathbb{L}f-f}.  \tag{4}
\end{equation}

This identity expresses the fact that the commutator with the Liouville
vector field measures the degree of homogeneity of the Hamiltonian function.

On the other hand, using the Leibniz rule for the Lie derivative, we have%
\begin{equation*}
\mathcal{L}_{\mathbb{L}}\left\{ \Omega (X_{f},X_{g})\right\} =\left( 
\mathcal{L}_{\mathbb{L}}\Omega \right) (X_{f},X_{g})-\Omega (\mathcal{L}_{%
\mathbb{L}}X_{f},X_{g})-\Omega (X_{f},\mathcal{L}_{\mathbb{L}}X_{g}).
\end{equation*}%
If we write the Poisson bracket on $M^{s}$ as $[f,g]_{\Omega }=\Omega
(X_{f},X_{g})$ then, taking into account equation (4), we obtain%
\begin{equation*}
\mathbb{L}[f,g]_{\Omega }=[\mathbb{L}f,g]_{\Omega }+[f,\mathbb{L}g]_{\Omega
}-[f,g]_{\Omega }.
\end{equation*}

As a consequence, the Poisson bracket on $(M^{s},\Omega )$ preserves the
space of homogeneous functions of degree one. This observation naturally
leads to the following definition.

\begin{definition}
Given $F,G\in C^{\infty }(M_{c}),$ we define the contact bracket as%
\begin{equation*}
\lbrack F,G]_{c}=\frac{1}{t}[tF,tG]_{\Omega },
\end{equation*}%
where $tF$ and $tG$ denote the corresponding homogeneous lifts to the
symplectic cover $M^{s}.$
\end{definition}

\noindent \textbf{Remark. }The stability of $1-$homogeneous functions, which
underlies the preceding definition, has profound topological consequences
that extend beyond the scope of this paper. In the approach to contact
Hamiltonian mechanics introduced earlier in \cite{GG}, where the
construction is carried out on a symplectic principal $GL(1,\mathbb{R)}-$%
principal bundle $\tau :P\rightarrow (M_{c},C)$, $1-$homogeneous functions
on $P$ are identified with sections of the dual line bundle $L_{P}^{\ast
}\rightarrow M$ where $L_{P}=P\times _{\mathbb{R}^{\ast }}\mathbb{R}%
\rightarrow M$ is the line bundle associated with the $\mathbb{R}^{\ast }-$%
principal bundle $P\rightarrow M$. This stability induces a Jacobi bracket
on the sections of $L_{P}^{\ast },$ thereby endowing the structure with that
of a local Kirillov manifold \cite{kirillov} or a Jacobi bundle \cite{marle}%
. The algebraic and geometric foundations of the Kirillov bracket are
presented, for instance, in \cite{BGG}, where its relation to Lie algebroids
is also emphasized. A further generalization of the Kirillov bracket to the
setting of Courant algebroids is developed in the previously cited work by
Grabowski \cite{Gra}, extending classical contact geometry to graded
manifolds and supermanifolds. This framework integrates graded contact
geometry with contact analogues of Courant algebroids, providing powerful
tools for the geometric description of dynamical systems with dissipation,
graded symmetries, or additional geometric constraints.

\bigskip

Although we will not explore these aspects further here, the expression in
Definition 5, which depends on the choice of contact coordinates underlying
(2), will later admit an intrinsic formulation in the cooriented case,
together with an explicit local coordinate expression.

\begin{proposition}
Let $X_{f}$ be the Hamiltonian vector field on the symplectic manifold $%
(M^{s},\Omega =d\omega )$ defined by%
\begin{equation*}
i_{X_{f}}\Omega =df.
\end{equation*}%
Then,

(a) $i_{X_{f}}\omega =-\mathbb{L}(f).$

(b) $L_{_{X_{f}}}\omega =0\Longleftrightarrow df$ it is a homogeneous $1-$%
form of degree $1$.
\end{proposition}

\begin{proof}
(a) Using the definition of the Liouville vector field $\mathbb{L,}$ we
compute%
\begin{equation*}
i_{X_{f}}\omega =i_{X_{f}}i_{\mathbb{L}}\Omega =-i_{\mathbb{L}}df=-\mathbb{L}%
(f).
\end{equation*}

(b) By the Cartan formula, 
\begin{equation*}
\mathcal{L}_{_{X_{f}}}\omega =\left( i_{_{X_{f}}}d+di_{_{X_{f}}}\right)
\omega =df-d\mathbb{L}(f)=df-\mathcal{L}_{\mathbb{L}}\left( df\right) ,
\end{equation*}%
whereby the statement of the Proposition.
\end{proof}

\begin{definition}
A vector field $X\in \mathfrak{X(}M^{s})$ verifying%
\begin{equation*}
\mathcal{L}_{X}\omega =0
\end{equation*}%
will be called an exact infinitesimal symplectomorphism on the symplectic
manifold $(M^{s},\Omega =d\omega ).$
\end{definition}

If $X\in \mathfrak{X(X}^{s})$ is an exact infinitesimal symplectomorphism,
then $\left[ X,\mathbb{L}\right] =0.$ In fact, we have%
\begin{eqnarray*}
i_{\left[ X,\mathbb{L}\right] }\Omega &=&\left[ L_{X},i_{\mathbb{L}}\right]
\Omega =\left( \mathcal{L}_{X}i_{\mathbb{L}}-i_{\mathbb{L}}\mathcal{L}%
_{X}\right) \Omega \\
&=&\mathcal{L}_{X}\omega =0.
\end{eqnarray*}%
By non-degeneracy of $\Omega ,$ this implies $\left[ X,\mathbb{L}\right] =0.$

\begin{proposition}
An exact infinitesimal symplectomorphism $X\in \mathfrak{X(}M^{s})$ is
projectable with respect to $\pi :M^{s}\rightarrow M_{c}$ and its $\pi -$%
projection $Y\in \mathfrak{X(}M_{c})$ is an infinitesimal contactomorphism
on $(M_{c},\mathcal{C}).$ Conversely, for every contactomorphism $Y$ on $%
M_{c}$, there exists an exact infinitesimal symplectomorphism $X\in 
\mathfrak{X(}M^{s})$ that projects to $Y.$
\end{proposition}

\begin{proof}
As functions on $M_{c}$ can be understood as $\mathbb{L}-$invarinat
functions on $M^{s}$, the relation $\left[ X,\mathbb{L}\right] =0$ implies
that the vector field $X$ determines a derivation $Y$ of the algebra $\pi
^{\ast }C^{\infty }(M_{c})$. Hence $X$ is $\pi -$projectable and determines
a vector field $Y$ on $M_{c}.$ With a slight abuse of notation and using the
local expression (3), we compute%
\begin{equation*}
0=\mathcal{L}_{X}\omega =\mathcal{L}_{X}(t\alpha _{U})=(Xt)\alpha _{U}+t%
\mathcal{L}_{Y}\alpha _{U}
\end{equation*}%
from which it follows that%
\begin{equation*}
\mathcal{L}_{Y}\alpha _{U}=-\frac{Xt}{t}\alpha _{U}.
\end{equation*}%
To see that $Xt/t$ is in fact a function on $M_{c},$ let us notice that the
relations $\mathbb{L}\mathcal{(}\log t)=1$ and $\left[ X,\mathbb{L}\right]
=0 $ imply $\mathbb{L}\mathcal{(}Xt/t)=0.$ Therefore $Y$ is an infinitesimal
contactomorphism.

Conversely, let $Y\in \mathfrak{X(}M_{c})$ be an infinitesimal
contactomorphism, so that locally%
\begin{equation*}
\mathcal{L}_{Y}\alpha _{U}=g\cdot \alpha _{U}
\end{equation*}%
for some function $g.$ Let $X\in \mathfrak{X(}M^{s})$ be a vector field $\pi
-$projectable onto $Y$ with $X(t)=0.$ Using the previous notation, we have%
\begin{equation*}
\mathcal{L}_{X-g\mathbb{L}}\omega =\mathcal{L}_{X-g\mathbb{L}}(t\cdot \alpha
_{U})=tg\cdot \alpha _{U}-tg\cdot \alpha _{U}=0.
\end{equation*}%
Thus, $X-g\mathbb{L}$ is a local exact infinitesimal symplectomorphism. The
uniqueness of the function $g$ on the above expression allows one to define $%
X-g\mathbb{L}$ globally on $M_{s}$ as an exact infinitesimal
symplectomorphism projecting onto $Y.$
\end{proof}

\begin{proposition}
Let $f\in C^{\infty }(M^{s})$ be a smooth function which is $1-$homogeneous, 
$\mathbb{L}(f)=f.$ Let $X_{f}$ denote the Hamiltonian vector field on $M^{s}$
defined by $i_{X_{f}}\Omega =df.$ Then, $X_{f}$ is $\pi -$projectable and in
local canonical coordinates $(t,z,x_{i},y_{j})$ it is given by%
\begin{align}
X_{f} &=
t\left( \frac{f}{t}\right) _{z}\frac{\partial }{\partial t}%
+\sum_{i=1}^{n}\left( \frac{f}{t}\right) _{y_{i}}\frac{\partial }{\partial
x_{i}}-\sum_{i=1}^{n}\left\{ \left( \frac{f}{t}\right) _{x_{i}}+y_{i}\left( 
\frac{f}{t}\right) _{z}\right\} \frac{\partial }{\partial y_{i}}  \notag \\
&\quad+\sum_{i=1}^{n}\left\{ y_{i}\left( \frac{f}{t}\right) _{y_{i}}-\frac{f}{t}%
\right\} \frac{\partial }{\partial z}  \tag{5}
\end{align}%
Moreover, the vector field $tX_{f/t}$ is also $\pi -$projectable, and its
projection is given by%
\begin{equation}
\pi _{\ast }\left( tX_{f/t}\right) =\sum_{i=1}^{n}\left( \frac{f}{t}\right)
_{y_{i}}\frac{\partial }{\partial x_{i}}-\sum_{i=1}^{n}\left\{ \left( \frac{f%
}{t}\right) _{x_{i}}+y_{i}\left( \frac{f}{t}\right) _{z}\right\} \frac{%
\partial }{\partial y_{i}}+\sum_{i=1}^{n}y_{i}\left( \frac{f}{t}\right)
_{y_{i}}\frac{\partial }{\partial z}.  \tag{6}
\end{equation}
\end{proposition}

\begin{proof}
By Proposition 6 we know that $L_{X_{f}}\omega =0.$ Proposition 8 then
implies that the vector field $X_{f}$ is $\pi -$projectable. Moreover, the
function $f/t$ $M^{s}$ satisfies $\mathbb{L}(f/t)=0$ and is therefore also $%
\pi -$projectable. Hence, formula (5) follows from a straightforward
computation using the Darboux expression%
\begin{equation*}
\Omega =dt\wedge dz+\sum_{i=1}^{n}d\left( -ty_{i}\right) \wedge dx_{i}.
\end{equation*}%
Furthermore, the identity 
\begin{equation}
X_{f}=tX_{f/t}+\frac{f}{t}X_{t}  \tag{7}
\end{equation}%
together with $X_{t}=-\partial /\partial z$ (that is, $i_{\partial /\partial
z}\Omega =-dt),$ and formula (5), yields expression (6).
\end{proof}

\begin{proposition}
In local canonical coordinates $(z,x_{i},y_{j})$, the contact bracket of two
smooth functions $F,G\in C^{\infty }(M_{c})$ is given by,%
\begin{equation}
\lbrack F,G]_{c}=\sum_{i}\left( \frac{\partial F}{\partial x_{i}}\frac{%
\partial G}{\partial y_{i}}-\frac{\partial F}{\partial y_{i}}\frac{\partial G%
}{\partial x_{i}}\right) +\sum_{i}y_{i}\left( \frac{\partial F}{\partial z}%
\frac{\partial G}{\partial y_{i}}-\frac{\partial F}{\partial y_{i}}\frac{%
\partial G}{\partial z}\right) +F\frac{\partial G}{\partial z}-G\frac{%
\partial F}{\partial z}.  \tag{8}
\end{equation}
\end{proposition}

\begin{proof}
By definition, the contact bracket is%
\begin{equation*}
\lbrack F,G]_{c}=\frac{1}{t}[tF,tG]_{\Omega }=-\frac{1}{t}X_{tF}(tG).
\end{equation*}%
A direct computation, using the explicit expression for the Hamiltonian
vector field $X_{tF}$ given in Proposition 9, yields precisely the local
coordinate formula in (8).
\end{proof}

\begin{corollary}
Let $(M_{c},C)$ be a cooriented contact manifold and $(M^{s},\Omega =d\omega
)$ its symplectic cover with Liouville vector field $\mathbb{L}$. Then,

(a) The Hamiltonian vector field $X_{t}$ associated with the $1-$homogeneous
function $t$ is $\pi -$projectable. We define the Reeb vector field $%
\mathcal{R}$ on $(M_{c},C)$ by $\pi _{\ast }(X_{t})=-\mathcal{R}.$

(b) For any $G\in C^{\infty }(M_{c})$, 
\begin{equation*}
\lbrack 1,G]_{c}=\mathcal{R}(G).
\end{equation*}
\end{corollary}

\smallskip

\noindent \textbf{Hamiltonian vector fields.} Hamiltonian vector fields
constitute one of the central notions in contact geometry and contact
Hamiltonian dynamics. Their introduction can be traced back to the
pioneering work of Lichnerowicz \cite{lich}, where the Lie algebraic
structure of infinitesimal automorphisms of a contact structure was studied
in depth. An overview of these ideas is provided in the work of Libermann
and Marle \cite{LM}. More recent presentations, motivated by applications to
dissipative dynamics, include the work of de Le\'{o}n and Lainz \cite{dLL}
and Bravetti \textit{et al.} \cite{BCT}.

From now on, we will assume that the contact sheaf $C$ of the contact
structure $(M_{c},C)$ is generated by a global section $\alpha .$ We will
then say that $\alpha $ is a contact $1-$form and we will refer to the pair $%
(M_{c},\alpha )$ as a contact manifold. The $1-$form $\alpha $ satisfies 
\begin{equation*}
\ker d\alpha \cap \ker \alpha =\{0\},
\end{equation*}%
in such a way that de restriction of $d\alpha $ to the distribution $\ker
\alpha $ is symplectic. The Reeb vector field of $(M_{c},\alpha )$ is the
unique vector field $\mathcal{R}\in \mathfrak{X}(M_{c})$ such that $\alpha (%
\mathcal{R})=1$ and $i_{\mathcal{R}}(d\alpha )=0.$

For every tangent vector $X\in TM_{c},$ $X\in T_{x}M_{c},$ $x\in M_{c},$ the
expression%
\begin{equation*}
X=\alpha (X)\mathcal{R+(}X\mathcal{-\alpha (}X\mathcal{)R)},
\end{equation*}%
provides the Whitney decomposition of vector bundles%
\begin{equation*}
TM_{c}=\ker d\alpha \oplus \ker \alpha .
\end{equation*}%
Usually the $1-$ rank bundle $\ker d\alpha ,$ is called the \textit{vertical
bundle} and the $2n-$rank bundle $\ker \alpha ,$ \textit{the horizontal
bundle} of the contact structure.

At this point, we introduce the dynamics on the contact system $%
(M_{c},\alpha )$ defined by a smooth function $H:M_{c}\rightarrow \mathbb{R}$%
. The Hamiltonian vector field $X_{H}$ is the unique vector field on $M_{c}$
satisfying the conditions:%
\begin{equation*}
\mathcal{L}_{X_{H}}\alpha =-\rho \alpha ,\text{ }(\rho \in C^{\infty
}(M_{c})),\text{ }\alpha (X_{H})=-H.
\end{equation*}%
It follows directly that $\rho =\mathcal{R}(H).$ Explicitly, in the
canonical coordinates $(z,x_{i},y_{j})$, the vector field $X_{H}~$takes the
form 
\begin{equation*}
X_{H}=\sum_{i=1}^{n}H_{y_{i}}\frac{\partial }{\partial x_{i}}%
-\sum_{i=1}^{n}\left\{ H_{x_{i}}+y_{i}H_{z}\right\} \frac{\partial }{%
\partial y_{i}}+\left( \sum_{i=1}^{n}y_{i}H_{y_{i}}-H\right) \frac{\partial 
}{\partial z}.
\end{equation*}%
Moreover, the Hamiltonian vector field $X_{H}$ 
\begin{equation*}
\mathcal{L}_{X_{H}}(H)=-\mathcal{R}(H)H
\end{equation*}%
and 
\begin{equation*}
\mathcal{L}_{X_{H}}\nu =-(n+1)\mathcal{R}(H)\nu ,
\end{equation*}%
where $\nu =\alpha \wedge (d\alpha )^{n}$ the contact volume element
associated with the contact form $\alpha .$

From now on, the triple $(M_{c},\alpha ,H)$ will be referred to as a contact
Hamiltonian system. We define the contact bracket of two smooth functions $F$
and $G$ on $M_{c}$ (see \cite{dLL}) as 
\begin{equation*}
\lbrack F,G]_{c}=i[X_{F},X_{G}]\alpha .
\end{equation*}%
Expanding this expression, we obtain%
\begin{eqnarray*}
i[X_{F},X_{G}]\alpha &=&\mathcal{L}_{X_{F}}i_{X_{G}}\alpha -i_{X_{G}}%
\mathcal{L}_{X_{F}}\alpha =-\mathcal{L}_{X_{F}}(G)+i_{X_{G}}\mathcal{R(}F%
\mathcal{)\alpha } \\
&=&-X_{F}(G)-G\mathcal{R(}F\mathcal{)}.
\end{eqnarray*}%
A straightforward computation in local canonical coordinates yields the
following explicit expression for the contact bracket:%
\begin{equation*}
\lbrack F,G]_{c}=\sum_{i}\left( \frac{\partial F}{\partial x_{i}}\frac{%
\partial G}{\partial y_{i}}-\frac{\partial F}{\partial y_{i}}\frac{\partial G%
}{\partial x_{i}}\right) +\sum_{i}y_{i}\left( \frac{\partial F}{\partial z}%
\frac{\partial G}{\partial y_{i}}-\frac{\partial F}{\partial y_{i}}\frac{%
\partial G}{\partial z}\right) +F\frac{\partial G}{\partial z}-G\frac{%
\partial F}{\partial z}.
\end{equation*}

This local expression is consistent with that obtained in Proposition 10,
confirming the compatibility between the intrinsic definition of the contact
bracket and its formulation via $1-$homogeneous functions on the symplectic
cover.

\noindent The contact Hamiltonian vector field $X_{F}$ admits the
decomposition 
\begin{equation*}
X_{F}=-F\mathcal{R+E}_{F}
\end{equation*}%
where $\mathcal{E}_{F}$ takes values in the horizontal distribution defined
by the contact structure $\left( M_{c},\alpha \right) $. This decomposition
corresponds to the Whitney splitting of $X_{F}$ distinguishing its vertical
component (along the Reeb field $\mathcal{R)}$ from its horizontal component
(tangent to the contact distribution).

By Proposition 9, the evolution vector field $\mathcal{E}_{F}$ is giving by 
\begin{equation*}
\mathcal{E}_{F}=\pi _{\ast }\left( tX_{F}^{\Omega }\right) ,
\end{equation*}%
where $X_{F}^{\Omega }$ is the Hamiltonian vector field of $F$ on $%
(M^{s},\Omega =d\omega ).$

\noindent \textbf{Remark.} The evolution vector field $\mathcal{E}_{F}$,
plays a fundamental role in modeling processes involving friction, including
isolated thermodynamic systems. For further details, see, for instance, \cite%
{SLLD}.

For $F,G\in C^{\infty }(M),$ we define the reduced contact bracket by%
\begin{equation*}
\lbrack F,G]_{c}^{r}=\mathcal{E}_{F}(G).
\end{equation*}

\begin{proposition}
Let $(M_{c},\alpha ,H)$ be a contact Hamiltonian system. The energy $H$ is
preserved by the flow of its corresponding evolution vector field $\mathcal{E%
}_{H}.$

However, $\mathcal{E}_{H}$ does not preserve the contact volume element $%
\mathcal{V}=\alpha \wedge (d\alpha )^{n}.$ In fact, the Lie derivative of $%
\mathcal{V}$ along $\mathcal{E}_{H}$ satisfies%
\begin{equation}
\mathcal{L}_{\mathcal{E}_{H}}\mathcal{V}=dH\wedge (d\alpha )^{n}-(n+1)%
\mathcal{R}(H)\mathcal{V.}  \tag{9}
\end{equation}%
As a consequence, if $\mathcal{R}(H)=0,$ then $\mathcal{E}_{H}$ does
preserve the contact volume element $\mathcal{V.}$
\end{proposition}

\begin{proof}
The first assertion follows directly from the fact that 
\begin{equation*}
\mathcal{L}_{\mathcal{E}_{H}}(H)=\mathcal{E}_{H}(H)=0.
\end{equation*}%
On the other hand, the following relations hold 
\begin{equation*}
\begin{array}{c}
\alpha (\mathcal{E}_{H})=0 \\ 
\mathcal{L}_{\mathcal{E}_{H}}(\alpha )=dH-\mathcal{R}(H)\alpha .%
\end{array}%
\end{equation*}%
A straightforward computation using these relations yields equation (9),
completing the proof.
\end{proof}

\begin{definition}
A submanifold $N\subset M_{c}$ is called isotropic if $\pi ^{-1}(N)$ is an
isotropic submanifold of the symplectic manifold $M^{s}$, and coisotropic if 
$\pi ^{-1}(N)$is coisotropic with respect to the same symplectic structure.
When $N$ is both isotropic and coisotropic, it is called a Legendrian
submanifold of $M_{c}.\medskip $
\end{definition}

\noindent \textbf{Remark.} It is interesting to put in perspective the
general frame of fibered Liouville manifolds, that is, a manifold ~$P$
endowed with a closed $2-$form $\Omega $ without zeros (and no restrictions
on its range) and with a free and regular action $\tau $ of the
multiplicative $\mathbb{R}^{\ast }$ such that%
\begin{equation*}
\tau _{t}^{\ast }\Omega =t\Omega ,\text{ }t\in \mathbb{R}^{\ast }.
\end{equation*}%
If we consider the field $Z,$%
\begin{equation*}
Z_{x}=\left. \frac{d}{dt}\tau _{t}(x)\right\vert _{t=1}
\end{equation*}%
then $L_{Z}\Omega =\Omega ,$ and hence $\Omega =d\left( i_{Z}\Omega \right) $%
.

In this way, the principal bundle $\pi :P\rightarrow P/\tau =M$ establishes
a bijective correspondence between integral submanifolds $Q\subset P$ of $%
\Omega $ that are invariant by $\tau _{t}$ and submanifolds $N$ that are
solutions of the Pfaffian system generated by the projection of the $1-$form 
$i_{Z}\Omega $ (see \cite{LM}).

\begin{proposition}
Let $N\subset M_{c}$ be a submanifold, and let $F_{1},...,F_{k}$ be local
defining functions of $N$ on an open set $U\subset M_{c}.$Let $\mathcal{E}%
_{F_{i}}$ $(1\leq i\leq k)$ denote the corresponding evolution vector
fields. Then%
\begin{eqnarray*}
N\text{ is coisotropic if and only if \ }\left\langle \mathcal{E}%
_{F_{i}}\right\rangle &\subseteq &T(N\cap U) \\
N\text{ is isotropic if an only if }T(N\cap U)\text{ } &\subseteq &\text{\ }%
\left\langle \mathcal{E}_{F_{i}}\right\rangle \\
N\text{ is legendrian if and only if \ }T(N\cap U)\text{ } &=&\text{\ }%
\left\langle \mathcal{E}_{F_{i}}\right\rangle .
\end{eqnarray*}
\end{proposition}

\begin{proof}
Consider the functions $f_{\alpha }=tF_{\alpha }$ $($for $1\leq \alpha \leq
m)$ as local generators of $\pi ^{-1}(N\cap U)$ on $\pi ^{-1}(U).$ By
Proposition 9, the corresponding Hamiltonian vector fields with respect to
the symplectic structure $(M^{s},\Omega )$ are given by%
\begin{equation}
X_{f_{\alpha }}=t\frac{\partial F_{\alpha }}{\partial z}\frac{\partial }{%
\partial t}+\sum_{i=1}^{n}\frac{\partial F_{\alpha }}{\partial y_{i}}\frac{%
\partial }{\partial x_{i}}-\sum_{i=1}^{n}\left\{ \frac{\partial F_{\alpha }}{%
\partial x_{i}}+y_{i}\frac{\partial F_{\alpha }}{\partial z}\right\} \frac{%
\partial }{\partial y_{i}}+\sum_{i=1}^{n}\left\{ y_{i}\frac{\partial
F_{\alpha }}{\partial y_{i}}-F_{\alpha }\right\} \frac{\partial }{\partial z}%
,\quad  \tag{10}
\end{equation}%
\linebreak for $1\leq \alpha \leq m.$

At every point $q\in \pi ^{-1}(N\cap U),$ these vector fields generate the
orthogonal subspace $\left( T_{q}(\pi ^{-1}(N\cap U))\right) ^{\bot _{\Omega
}}.$ Taking into account the relative positions of the spaces $T_{q}(\pi
^{-1}(N\cap U))$ and $\left( T_{q}(\pi ^{-1}(N\cap U))\right) ^{\bot
_{\Omega }}$ according on whether $N$ $\ $is isotropic or coisotropic and
noting that the projections by $\pi $ of the vector fields in (10), when
restricted to $N\cap U,$ coincide with the evolution vector fields $\mathcal{%
E}_{F_{i}}$ $($for $1\leq \alpha \leq m),$ the necessity of the conditions
in the statement follows.

To establish sufficiency, note first that the vector field $\partial
/\partial t$ is everywhere tangent to $T(\pi ^{-1}(N\cap U)).$ Therefore, in
order to verify the inclusion 
\begin{equation*}
\left( T_{q}(\pi ^{-1}(N\cap U))\right) ^{\bot _{\Omega }}\subset T(\pi
^{-1}(N\cap U))
\end{equation*}%
it is enough to check that the Hamiltonian vector fields $\left\langle
X_{f_{\alpha }}\right\rangle $ which generate the symplectic orthogonal
subspace, are tangent to $T(\pi ^{-1}(N\cap U))$.

Similarly, in the isotropic case, to verify that $T(\pi ^{-1}(N\cap
U))\subseteq \left\langle X_{f_{\alpha }}\right\rangle ,$ it suffices to
check that $T(\pi ^{-1}(N\cap U))\subseteq \left\langle \mathcal{E}%
_{F_{i}}\right\rangle ,$ since the vertical direction $\partial /\partial t$
is already contained in $T(\pi ^{-1}(N\cap U)).$\smallskip
\end{proof}

\noindent \textbf{Remark. }In this context, it is worth considering
deformation-theoretic aspects of contact geometry. In particular, rigidity
properties of integral coisotropic submanifolds have been established in 
\cite{T}, indicating that their infinitesimal deformations are strongly
constrained by the ambient contact structure.

\begin{proposition}
The sheaf of ideals of a coisotropic manifold $N\subset M_{c}$ is stable by
the reduced contact Bracket.
\end{proposition}

\begin{proof}
Let $F,G\in C^{\infty }(M_{c})$ be such that $F|_{N}=0$ and $G|_{N}=0$. By
definition of the reduced contact bracket, 
\begin{equation*}
\lbrack F,G]_{c}^{r}=E_{F}(G).
\end{equation*}%
Using the relation $E_{F}=\pi _{\ast }(tX_{F}^{\Omega })$, we obtain 
\begin{equation*}
\lbrack F,G]_{c}^{r}\big|_{N}=\pi _{\ast }(tX_{F}^{\Omega })(G)\big|%
_{N}=tX_{F}^{\Omega }(G)\big|_{\pi ^{-1}(N)}.
\end{equation*}%
Since $N$ is coisotropic, its inverse image $\pi ^{-1}(N)\subset M^{s}$ is a
coisotropic submanifold of the symplectic manifold $(M^{s},\Omega )$.
Therefore, the set of functions vanishing on $\pi ^{-1}(N)$ is stable under
the Poisson bracket, and we have 
\begin{equation*}
X_{F}^{\Omega }(G)\big|_{\pi ^{-1}(N)}=[F,G]_{\Omega }\big|_{\pi ^{-1}(N)}=0,
\end{equation*}%
where $[\cdot ,\cdot ]_{\Omega }$ denotes, as usual, the Poisson bracket
induced by $\Omega $. It follows that 
\begin{equation*}
\lbrack F,G]_{c}^{r}\big|_{N}=0,
\end{equation*}%
which proves the claim.

\smallskip
\end{proof}

\noindent \textbf{Remark. }Let $N\subset M_{c}$ be a submanifold. De Le\'{o}%
n and Lainz \cite{dLL} define $x\in N$ as horizontal point if $%
T_{x}N=T_{x}N\cap \mathcal{H}_{x},$ that is, $\left. \alpha \right\vert
T_{x}N=0.$ Here we will say that $p\in N$ is an ordinary point if it is not
horizontal.

Given a Pfaff system $\mathcal{F}$ on $N,$ its characteristic system $C(%
\mathcal{F})\subset \mathcal{H}$ is the set of fields on $N$, horizontal for
the contact structure, that stabilize the system $\mathcal{F}:$%
\begin{equation*}
C(\mathcal{F})=\{X\in T_{p}N\cap \mathcal{H}_{p}:L_{X}w\in \mathcal{F}_{p},%
\text{ with }w_{p}\in \mathcal{F}_{p},\text{ }p\in N\}.
\end{equation*}%
Now consider a coisotropic submanifold $N\subset M_{c}$ and let $\left\{
F_{1},...,F_{m}\right\} $ be a set of smooth generating functions for the
sheaf of ideals of $N$ on a local chart $U\subset M_{c}$ that contains no
horizontal points.

In fact, we can assume that $U\subset M_{c}$ is diffeomorphic to an open set
of $\mathbb{R}^{2n+1}$ coordinate by $x_{1},...,x_{n},z,y_{1},...,y_{n,\text{
}}$ in which the contact form is written as $\alpha =dz-\sum y_{i}dx_{i}.$
Let us consider on $U$ the Pfaff system%
\begin{equation*}
(\alpha ,dF_{1},...,dF_{m}),
\end{equation*}%
which we assume to be linearly independent at every point. This condition is
equivalent to the linear independence of the corresponding evolution vector
fields $\mathcal{E}_{F_{1}},...,\mathcal{E}_{F_{m}}$.

To say that a vector field $X\in \mathcal{H}$ preserves the Pfaffian system $%
(\alpha ,dF_{1})$ means, by direct computation, $X$ must be proportional to
the vector field%
\begin{equation*}
\mathcal{E}_{F_{1}}=\sum_{i=1}^{n}F_{1y_{i}}\frac{\partial }{\partial x_{i}}%
-\sum_{i=1}^{n}\left\{ F_{1x_{i}}+y_{i}F_{1z}\right\} \frac{\partial }{%
\partial y_{i}}+\sum_{i=1}^{n}y_{i}F_{1y_{i}}\frac{\partial }{\partial z}.
\end{equation*}%
This follows from the identity $L_{\mathcal{E}_{F_{1}}}\alpha
=dF_{1}-F_{1z}\alpha .$ In this way, $\mathcal{E}_{F_{1}}\in C(\alpha
,dF_{1})$ and in fact $\mathcal{E}_{F_{1}}\in C(w,dF_{1},...,dF_{m}),$ since
by the coisotropic condition on $N$, we have $\mathcal{E}_{F_{1}}(F_{j})=0,$ 
$j=1,...,m.$ Since this argument is constructive, we have%
\begin{equation*}
C(\alpha ,dF_{1},...,dF_{m})=\left\langle \mathcal{E}_{F_{1}},...,\mathcal{E}%
_{F_{m}}\right\rangle .
\end{equation*}%
The vector fields $\mathcal{E}_{F_{1}},...,\mathcal{E}_{F_{m}}$ are linearly
independent if and only if, as one can verify, the $1-$forms $\alpha
,dF_{1},...,dF_{m}$ are linearly independent on $U$. Moreover, the
restriction of the vector fields $\{\mathcal{E}_{F_{i}}\}$ $(i=1,...,m),$
generates the characteristic system of the Pfaff system defined by the
specializatio $\overline{\alpha }$ of the contact form to $U\cap N$.
Consequently, if $U\cap N$ consists of ordinary points, the characteristic
system of the Pfaff system $\left. \overline{\alpha }\right\vert _{U}$ has
rank $m.$

Thus, since $\dim N=2n+1-m$, in a sufficiently small neighborhood of $N$
---which we continue to denote by $U,$ ---by the very framework of Darboux
theorem, there exist $2(n-m)+1$ functions $\delta ,\beta _{i},\gamma _{i}$ $%
(i=1,...,n-m)$ on $U$ such that the form $\overline{\alpha }$ in $U$ is
expressed as%
\begin{equation*}
d\delta -\sum_{i=1}^{n-m}\beta _{i}d\gamma _{i}.
\end{equation*}%
In this way, $\delta =k,$ $\gamma _{i}=k_{i}$ for arbitrary constants $%
k,k_{i},$ $(i=1,...,n-m).$ provide a Legendrian foliation in $V$ of the
coisotropic manifold $N$ (see \cite{KT} or \cite{L} for discussions of
Legendre foliations in the context of contact complete integrability).

\section{Hamilton-Jacobi theory}

In the standard setting of classical mechanics, a Lagrangian system is
defined by a function $L$ on the tangent bundle $TQ$ of the configuration
manifold, governing its dynamics though the Euler-Lagrange equations.
However, this requirement can be overly restrictive: certain physical
situations or constraints lie beyond the reach of the Lagrangian framework.
A notable example occurs when generalized forces cannot be derived from a
potential, in which case no suitable Lagrangian exists and the action
principle cannot be applied.

In this context, a notable advancement in recent years has enabled the
extension of the Lagrangian framework to dissipative mechanical systems, as
proposed in \cite{LLa}. This result makes the formalism relevant to new
problems and reveals new geometric structures of physical interest \cite%
{GGMRRx}. For Lagrangian approach to dissipative dynamics, see \cite{Ga,MR},
while the geometric and thermodynamic foundations of dissipation are
discussed in \cite{Raz}.

The concept of a dissipative Hamiltonian $H$ is formulated on the extended
phase space $\left( T^{\ast }Q\times \mathbb{R},\eta \right) $ equipped with
its canonical contact form $\eta .$ In this framework, contact geometry
emerges as a fundamental tool for the description of mechanical systems with
dissipation, offering a natural extension of symplectic methods to settings
where energy is no longer conserved. This geometric perspective unifies
conservative and dissipative dynamics and also provides a rigorous
foundation studying integrability, stability, and variational principles in
non-conservative contexts. Notable paradigmatic examples of this approach
can be found in \cite{Brr,BCT,dLL,GGMRRx}.

Particularly striking interpretations of contact geometry, where contact
transformations describe thermodynamic processes, contact Hamiltonians serve
as thermodynamic potentials, and Legendre submanifolds arise as attractors
of contact Hamiltonian vector fields, thereby representing equilibrium
states in thermodynamic variables, are detailed, for instance, in \cite%
{Br,Go,Gr,Mr2}.

Motivated by these developments, recent advances have extended
Hamilton--Jacobi theory to contact Hamiltonian systems, establishing a
geometric framework for the associated equations and analyzing their
integrability conditions.

Before proceeding, it is essential to recall the Hamilton--Jacobi problem in
its classical geometric setting.

Within the symplectic framework of Hamiltonian dynamical systems, the main
objective is to characterize Lagrangian submanifolds $L\subset T^{\ast }Q$,
that is, $n-$dimensional subvarieties of $T^{\ast }Q$ on which the canonical
symplectic form $w_{2}$ vanishes, such that the Hamiltonian vector field $%
X_{H}\in \mathfrak{X}(T^{\ast }Q)$ is tangent to $L.$ This geometric
formulation relies on the following fundamental result.

\begin{theorem}
Let \ $(X,\omega _{2})$ be a symplectic manifold, let us consider the
dynamic equation%
\begin{equation*}
i_{Z_{H}}\omega _{2}=dH,
\end{equation*}%
and let $L\subset X$ be a Lagrangian submanifold. Then the vector field $%
Z_{H}$ is tangent to $L$ if and only if $\left. H\right\vert _{L}$ is
constant.
\end{theorem}

For a comprehensive understanding of this result, readers may consult both
theoretical and applied references. Foundational aspects are rigorously
addressed in the seminal work of \cite{AM}, while modern geometric
perspectives are developed in \cite{LMV}. Of particular relevance is the
treatment presented in \cite{CGMMMR}, which provides critical insights into
the geometric foundations of this topic.

\bigskip

\begin{definition}
Let us now consider the symplectic cover $(M^{s},\Omega =d\omega )$ of the
contact structure $(M_{c},\eta )$. A submanifold $L\subset M^{s}$ is called
a Liouville submanifold if it is invariant under the flow of the Liouville
vector field $\mathbb{L}.$
\end{definition}

\begin{theorem}[Hamilton-Jacobi for Liouville submanifolds]
Let $H$ be a $1-$homogeneous smooth function on $M^{s}$, that is, $\mathbb{L}%
(H)=H.$ Consider the Hamilton equation%
\begin{equation}
i_{Z_{H}}\Omega =dH.  \tag{11}
\end{equation}%
where $Z_{H}$ is the Hamiltonian vector field associated to $H.$ Let $%
L\subset M^{s}$ be a Liouville Lagrangian submanifold.
\end{theorem}

Then,

(a) $H\equiv -\omega (Z_{H})\func{mod}const.$ on each connected component of 
$M^{s}.$

(b) $\left. \omega \right\vert _{L}=0.$

(c) \ $Z_{H}\in \mathfrak{X}(L)\Leftrightarrow \left. H\right\vert _{L}=0%
\func{mod}$ const. on each connected component of $M^{s}.$

\begin{proof}
(a) Since by Proposition 6, $\mathcal{L}_{Z_{H}}\omega =0,$ we have%
\begin{equation*}
i_{Z_{H}}d\omega +di_{Z_{H}}\omega =0
\end{equation*}%
consequently%
\begin{equation*}
dH=-d\left[ \omega (Z_{H})\right] .
\end{equation*}%
(b) Since $L$ is a Lagrangian submanifold and the Liouville vector field $%
\mathbb{L}$ is tangent to $L,$ we have $\left. \omega \right\vert
_{L}=\left. i_{\mathbb{L}}\Omega \right\vert _{L}=0.$

\noindent (c) If $Z_{H}\in \mathfrak{X}(L)$ it follows from (11) that $%
\left. H\right\vert _{L}\equiv $ const. Conversely, if $\left. H\right\vert
_{L}\equiv $ const., then by (11) we have $Z_{H}\in \mathfrak{X}(L)^{\bot }$
and hence $Z_{H}\in \mathfrak{X}(L)$ by the Lagrangian condition on $%
L.\medskip $
\end{proof}

A parallel geometric framework, bridging the contact setting with the
classical Theorem 16, is precisely formulated in the following theorem.

\begin{theorem}
Let $(M_{c},\eta ,H)$ be a contact Hamiltonian system, and let $\mathcal{E}%
_{H}$ denote the evolution vector field determined by the contact structure $%
\eta $ and the Hamiltonian function $H$. For a Legendrian submanifold $%
L\subset M_{c}$, the following conditions are equivalent:

\noindent (i) The vector field $\mathcal{E}_{H}$ is tangent to $L$.

\noindent (ii) The restriction of $H$ to $L$ is constant.
\end{theorem}

\begin{proof}
Consider the symplectic covering $\pi \colon (M^{s},\Omega )\rightarrow
(M_{c},\eta )$. If the Hamiltonian function $H$ is constant on $L$, then it
is also constant on $\pi ^{-1}(L)$ when considered as a function on $M^{s}$.
Consequently, by the classical Hamilton--Jacobi theorem, the Hamiltonian
vector field $X_{H}^{\Omega }$ is tangent to $\pi ^{-1}(L)$. By
Proposition~9, the vector field $t\,X_{H}^{\Omega }$ is $\pi $-projectable
and satisfies 
\begin{equation*}
\pi _{\ast }(t\,X_{H}^{\Omega })=\mathcal{E}_{H}.
\end{equation*}%
Therefore, $\mathcal{E}_{H}$ is tangent to $L$.

Conversely, assume that $\mathcal{E}_{H}$ is tangent to $L$. Since $\mathcal{%
E}_{H}=\pi _{\ast }(t\,X_{H}^{\Omega })$, it follows that the vector field $%
t\,X_{H}^{\Omega }$ is tangent to $\pi ^{-1}(L)$. Indeed, let $\{g_{i}\}$ be
a set of local generators of the sheaf of ideals defining $L$. Then the
pullbacks $\{g_{i}\circ \pi \}$ generate the sheaf of ideals defining $\pi
^{-1}(L)$. For each such generator, we have 
\begin{equation*}
(t\,X_{H}^{\Omega })(g_{i}\circ \pi )=\pi _{\ast }(t\,X_{H}^{\Omega
})(g_{i})=\mathcal{E}_{H}(g_{i})=0.
\end{equation*}%
By the classical Hamilton--Jacobi theorem, it follows that $H$ is constant
on $\pi ^{-1}(L)$ and, consequently, on $L$.\vspace{1em}
\end{proof}

\noindent \textbullet\ After some initial steps, such as those presented in 
\cite{LS}, (see also the review \cite{LLreview}) a formulation of a
Hamilton-Jacobi theory adapted to contact Hamiltonian systems can be found
in \cite{LLM}. Let us now consider the phase space $T^{\ast }Q\times \mathbb{%
R,}$ endowed with the contact structure defined by the contact form: $\eta
=dz-\theta $ where $\theta $ denotes the canonical $1-$form on $T^{\ast }Q.$
Let $H:T^{\ast }Q\times \mathbb{R\rightarrow R}$ be a Hamiltonian function.
As in the case of classical Hamilton-Jacobi theory, the geometric framework
is determined by sections from the configuration manifold into the phase
space $T^{\ast }Q\times \mathbb{R}.$ One should note how \cite{LLM} develops
the computational crystallization of the geometric framework underlying our
Theorem 18. In this regard, it is by now a classical result that, on an open
subset $U\subset Q$ with local coordinates $q^{1},...,q^{n}$ a local section 
$U\rightarrow T^{\ast }Q\times \mathbb{R}$ of the projection $T^{\ast
}Q\times \mathbb{R\rightarrow }Q$ defines a Legendrian submanifold if there
exists a smooth function $S$ on $U$ such that%
\begin{equation*}
z=S(q^{1},...,q^{n}),\text{ and }p_{i}=\frac{\partial S}{\partial q^{i}}.
\end{equation*}%
With a slight abuse of notation, we denote this section by%
\begin{equation*}
S:U\rightarrow T^{\ast }Q\times \mathbb{R}:(q^{i})\longmapsto \left( q^{i},%
\frac{\partial S}{\partial q^{i}},S\right) .
\end{equation*}%
Considering the evolution vector field%
\begin{equation*}
\mathcal{E}_{H}=\frac{\partial H}{\partial p_{i}}\frac{\partial }{\partial
q^{i}}-\left( p_{i}\frac{\partial H}{\partial z}+\frac{\partial H}{\partial
q^{i}}\right) \frac{\partial }{\partial p_{i}}+p_{i}\frac{\partial H}{%
\partial p_{i}}\frac{\partial }{\partial z},
\end{equation*}%
and imposing that it be tangent to the submanifold $\func{Im}S,$ we obtain%
\begin{eqnarray*}
\mathcal{E}_{H}\left( \frac{\partial S}{\partial q_{k}}(t,q)-p_{k}\right)
&=&\left( \sum_{i=1}^{n}\frac{\partial H}{\partial p_{i}}\frac{\partial }{%
\partial q_{i}}-\left( p_{i}\frac{\partial H}{\partial z}+\frac{\partial H}{%
\partial q^{i}}\right) \frac{\partial }{\partial p_{i}}\right) \left( \frac{%
\partial S}{\partial q_{k}}(t,q)-p_{k}\right) \\
&=&\frac{d}{dq_{k}}\left( \left. H\right\vert _{\func{Im}S}\right) =0\quad
(1\leq k\leq n),
\end{eqnarray*}%
which is equivalent to requiring that $\left. H\right\vert _{\func{Im}S}$ be
constant, as guaranteed by Theorem 18. This equivalence leads to the
following key result \cite{LLM}:

\begin{proposition}
Let $S$ be a local section of the projection $T^{\ast }Q\times \mathbb{%
R\rightarrow }Q$ such that $\func{Im}S$ is a Legendrian submanifold of $%
(T^{\ast }Q\times \mathbb{R,}\eta ).$ Then the evolution vector field $%
\mathcal{E}_{H}$ is tangent to $\func{Im}S$ if and only if $\left.
H\right\vert _{\func{Im}S}$ is constant.
\end{proposition}

This same article \cite{LLM} presents another formulation of the
Hamilton-Jacobi problem by seeking invariant submanifolds of the contact
Hamiltonian vector field $X_{H}$ through sections $\gamma $ of the
projection $T^{\ast }Q\times \mathbb{R\rightarrow }Q\times \mathbb{R.}$
Writing in local coordinates 
\begin{equation*}
\gamma :(q^{i},z)\mapsto (q^{i},\gamma _{i}(q^{i},z),z)
\end{equation*}%
the tangency condition of $X_{H}$ to $\func{Im}\gamma $ together with the
condition that $\func{Im}\gamma $ be coisotropic and foliated by $T^{\ast
}Q- $Lagrangian leaves of the form $(q^{i},\gamma _{i}(q^{i},z_{0})$ leads
to the Hamilton-Jacobi equation:%
\begin{equation*}
\frac{\partial H}{\partial q^{j}}+\frac{\partial H}{\partial p_{i}}\frac{%
\partial \gamma _{i}}{\partial q^{j}}+\gamma _{j}\left( \frac{\partial H}{%
\partial z}+\frac{\partial H}{\partial p_{i}}\frac{\partial \gamma _{i}}{%
\partial z}\right) -H\frac{\partial \gamma _{j}}{\partial z}=0.
\end{equation*}%
Thus, this provides a precise geometric characterization of the
Hamilton--Jacobi problem for a specific class of coisotropic submanifolds in
the contact setting. We will return later to the geometric framework
underlying this proposal.

\medskip

\noindent \textbullet\ The study of time-dependent contact Hamiltonian
systems and their integrability naturally extends the autonomous
non-conservative framework to non-autonomous dynamics. Indeed,
time-dependent contact Hamiltonian systems \cite{LGGMR,GLR,RT} introduce the
notion of cocontact manifolds. Such a manifold $M$ has dimension $2n+2$ and
carries two $1-$forms $(\tau ,\eta )$ with $d\tau =0$ and volume element $%
\tau \wedge \eta \wedge (d\eta )^{n}.$ A standard local model is the bundle $%
\mathbb{R}_{t}\times T^{\ast }Q\times \mathbb{R}$ where the cocontact
structure is given by $\tau =dt$ and by the canonical contact form $\eta $
on $T^{\ast }Q\times \mathbb{R}$.

Early contributions to a Hamilton--Jacobi theory for time-dependent contact
Hamiltonian systems appeared in \cite{LS}, and led to a comprehensive
formulation in \cite{LLLa}. The geometric framework relies on the fibration 
\begin{equation*}
\pi _{Q}^{t}:\mathbb{R}_{t}\times T^{\ast }Q\times \mathbb{R\rightarrow R}%
_{t}\times Q.
\end{equation*}%
Here, we refine the geometric condition that yields the classical
non-autonomous Hamilton--Jacobi equation in dissipative systems.
Specifically, in \cite{LLLa}, sections $\gamma $ of $\pi $ considered such
that $\gamma (t,)$ defines a time-dependent Legendrian submanifold in $%
T^{\ast }Q\times \mathbb{R}.$

The dynamics determined by a Hamiltonian function $H:\mathbb{R}_{t}\times
T^{\ast }Q\times \mathbb{R}\rightarrow \mathbb{R}$ on this phase space is
governed by the vector field $X_{H}$, characterized by%
\begin{equation*}
\tau (X_{H})=1,\quad \eta (X_{H})=-H,\quad i_{X_{H}}d\eta =dH-H_{t}\tau
-H_{z}\eta .
\end{equation*}%
In local Darboux coordinates $(t,q^{i},p_{j},z)$, this vector field takes
the form 
\begin{equation*}
X_{H}=\frac{\partial }{\partial t}+\frac{\partial H}{\partial p_{i}}\frac{%
\partial }{\partial q^{i}}-\left( p_{i}\frac{\partial H}{\partial z}+\frac{%
\partial H}{\partial q^{i}}\right) \frac{\partial }{\partial p_{i}}+\left(
p_{i}\frac{\partial H}{\partial p_{i}}-H\right) \frac{\partial }{\partial z}.
\end{equation*}%
Analogously to the autonomous formulation, there exists a locally defined
generating function $S(q,t)$ associated with a time-dependent Legendrian
submanifold $\gamma (t,\cdot )$ in $T^{\ast }Q\times \mathbb{R}$, such that
the section $\gamma $ can be written as follows (see also \cite{LLLa}): 
\begin{equation*}
\gamma :\mathbb{R}_{t}\times Q\mathbb{\rightarrow R}_{t}\times T^{\ast
}Q\times \mathbb{R}:(t,q)\mapsto \left( t,q^{i},\frac{\partial S}{\partial
q^{i}},S\right) .
\end{equation*}%
The tangency condition $\left. X_{H}\right\vert _{\func{Im}\gamma }\in 
\mathfrak{X}(\func{Im}\gamma )$, then implies%
\begin{equation*}
d_{Q}(H\circ \func{Im}\gamma )+d_{Q}(\frac{\partial S}{\partial t})=0
\end{equation*}%
Consequently, we obtain%
\begin{equation*}
\left. H\right\vert _{\func{Im}\gamma }+\frac{\partial S}{\partial t}%
(t,q)=g(t)
\end{equation*}%
for some function $g.$ By redefining $S(q,t)\longmapsto S(q,t)-\int g(t),$
this simplifies to the standard form

\begin{equation*}
\left. H\right\vert _{\func{Im}S}+\frac{\partial S}{\partial t}(t,q)=0,
\end{equation*}%
which enables complete integration of the Hamiltonian flow via the
generating function $S$ along $\func{Im}\gamma ,$ thereby generalizing the
classical Hamilton--Jacobi equation to time-dependent contact systems while
preserving integrability even in the presence of dissipative terms.\vspace{%
1em}

\noindent \textbullet\ In the same article, \cite{LLLa}, de Le\'{o}n \textit{%
et al.} also address some limitations of the previously discussed
action-independent framework. These include the requirement for complete
solutions depending on $n+1$ parameters (as opposed to the usual $n$ in
symplectic systems), and the restriction to zero-energy levels in the
autonomous case. To handle these issues, the authors present a second,
alternative method in Hamilton--Jacobi theory that works for contact
Hamiltonian systems, covering both autonomous and non-autonomous cases. The
approach consists in considering sections of the fibration%
\begin{equation*}
\pi _{Q}^{t,z}:\mathbb{R}\times T^{\ast }Q\times \mathbb{R}\rightarrow 
\mathbb{R}\times Q\times \mathbb{R}
\end{equation*}%
which are locally given by%
\begin{equation*}
\gamma (t,q,z)=(t,q,\gamma _{i}(t,q,z),z).
\end{equation*}%
Under the key condition that the image of $\gamma $ is a coisotropic
submanifold (rather than Legendrian, as in the previous approach), the
requirement that $X_{H}$ be tangent to $\func{Im}\gamma $ leads to the
equation%
\begin{equation*}
d_{Q}(H\circ \gamma )+\frac{\partial }{\partial z}(H\circ \gamma )\gamma +%
\mathcal{L}_{R_{t}}\gamma =(H\circ \gamma )\mathcal{L}_{\frac{\partial }{%
\partial z}}\gamma ,
\end{equation*}%
which is referred to as the action-dependent Hamilton--Jacobi equation for $(%
\mathbb{R}\times T^{\ast }Q\times \mathbb{R},\tau ,\eta ,H).$

At this stage, a brief reflection on the variational principle underlying
the formulation of a dissipative Hamiltonian theory is in order.

A remarkable advance in recent years has been the development of a
Lagrangian formalism capable of describing dissipative mechanical systems.
Within this approach, the Lagrangian is not limited to a function on the
tangent bundle, but rather takes the form $L\mathcal{(}q,\overset{\cdot }{q}%
,z):TQ\times \mathbb{R}\rightarrow \mathbb{R}$ where the extra variable $z$
(or a function of it) is used to encode frictional effects, energy losses,
or couplings with a thermal bath or other external environments.

This extension naturally gives rise to the geometric framework of contact
Lagrangian systems where the dynamics of dissipative processes are encoded
in the contact structure of the phase space. A key step is to formulate the
appropriate variational principle that governs such systems, ensuring that
the resulting equations capture both the conservative and dissipative
aspects of the dynamics.

Let us consider curves $\gamma :[0,T]\rightarrow Q$ in the configuration
space with fixed boundary conditions. To each such curve, we associate a
real-valued function $z_{\gamma }(t):[0,T]\rightarrow \mathbb{R}$ determined
by the differential relation%
\begin{equation*}
\overset{\cdot }{z}(t)=L\left( \gamma (t),\overset{\cdot }{\gamma }%
(t),z(t)\right) ,
\end{equation*}%
subject to the initial condition $z_{\gamma }(0)=z_{0}.$ The corresponding
action functional is then given by the terminal value $z_{\gamma }(T).$ One
proves that the trajectory $z_{\gamma }(t)$ is stationary with respect to
this action precisely when it satisfies the Euler--Lagrange--Herglotz
equations%
\begin{equation*}
\frac{\partial L}{\partial q^{i}}-\frac{d}{dt}\frac{\partial L}{\partial 
\dot{q}^{i}}+\frac{\partial L}{\partial z}\cdot \frac{\partial L}{\partial 
\dot{q}^{i}}=0.
\end{equation*}%
(see \cite{LLa} and \cite{VBS}). Let $L:TQ\times \mathbb{R\rightarrow R}$ be
a regular Lagrangian, meaning that the Hessian matrix $(\partial
^{2}L/\partial \overset{\cdot }{q}^{i}\partial \overset{\cdot }{q}^{j})$ is
nonsingular. Under this assumption, one introduces the map 
\begin{equation*}
TQ\times \mathbb{R\rightarrow }T^{\ast }Q\times \mathbb{R}:(q,\overset{\cdot 
}{q},z)\mapsto (q,p,z)
\end{equation*}%
where \textit{the conjugate momentum} is given by $p=\partial L/\partial 
\dot{q}.$ In this framework, the Hamiltonian function $H:T^{\ast }Q\times 
\mathbb{R}\rightarrow \mathbb{R}$ is defined as%
\begin{equation*}
H\mathcal{=}\dot{q}^{i}\frac{\partial L}{\partial \dot{q}^{i}}-L
\end{equation*}%
where the velocities $\dot{q}^{i}$ are expressed in terms of the coordinates 
$(q^{i},p_{j},z)$ by (locally) solving the relations $p_{j}=\partial
L/\partial \dot{q}^{j}$ as guaranteed by the regularity of the Lagrangian.

In analogy with classical Hamilton--Jacobi theory, one would expect the
geometric setting to arise from sections of the projection from the
configuration manifold to the extended phase space $\mathbb{R}_{t}\mathbb{%
\times }T^{\ast }Q\times \mathbb{R}.$ Traditionally, this structure
originates from the variational nature of the problem: in our case, within
the canonical Hamiltonian description of contact dynamics, one formally
writes 
\begin{equation}
\overset{\cdot }{z}(t)=\mathcal{L}\left( t,q(t),\overset{\cdot }{q}%
(t),z(t)\right) =p\overset{\cdot }{q}-H  \tag{12}
\end{equation}

However, a geometric framework whose origin lies in the variational
character of the problem should instead provide a local representation of
the form $z=S(q,t)$, capable of encoding the dynamics expressed in equation
(12).

These considerations highlight that the approach in \cite{LLLa}, while not
grounded in a traditional variational principle, successfully uses geometric
structures, such as coisotropy and foliation invariance, to develop a
coisotropic Hamilton--Jacobi theory. This framework effectively generates $%
n- $parameter solution families and extends the theory reach to non-zero
energy levels.\vspace{1em}

\noindent \textbullet\ Implicit Hamiltonian dynamics is introduced in the
canonical symplectic manifold $T^{\ast }Q$ by replacing the Lagrangian
submanifold defined by a Hamiltonian vector field $X_{H}$ with an arbitrary
Lagrangian submanifold \cite{ELSa,ELSb}. This generalization is naturally
described in terms of a Morse family $F=F(q^{i},p_{i},\lambda ^{a})$, which
acts as a generating function for such a submanifold. The resulting
Lagrangian submanifold $\mathcal{B}\subset T(T^{\ast }Q)$ is explicitly
given by: 
\begin{equation*}
\mathcal{B}=\left\{ \left( q^{i},p_{i};\frac{\partial F}{\partial p_{i}},-%
\frac{\partial F}{\partial q^{i}}\right) \in T(T^{\ast }Q)\quad |\;\frac{%
\partial F}{\partial \lambda ^{a}}=0\right\} .
\end{equation*}%
As in the classical explicit Hamilton--Jacobi theorem, for a closed $1-$form 
$\gamma $ defining a Lagrangian section of $T^{\ast }Q$, there exists a
fundamental relationship between suitable tangency conditions of the
restriction $\mathcal{B}|_{\func{Im}(\gamma )}$ and the exact differential
associated with the Morse family $F$ generating $\mathcal{B}$ (see \cite%
{ELSa,ELSb}).

In \cite{ELLS}, implicit contact Hamiltonian dynamics is introduced on $%
T(T^{\ast }Q)\times \mathbb{R}$ by endowing this space with a contact $1$%
-form $\eta ^{T}$ constructed from the canonical contact form $\eta $ on $%
T^{\ast }Q\times \mathbb{R}$. A Morse family $E$ generates a Legendrian
submanifold $N\subset T(T^{\ast }Q)\times \mathbb{R}$. Then, for sections $%
\gamma \colon Q\times \mathbb{R}\rightarrow T^{\ast }Q\times \mathbb{R}$
induced by a (possibly time-dependent) $1$-form on $Q$, there exists a
fundamental relationship between the tangency conditions of $N$ and the
exterior differential of the Morse family $E$ with respect to the canonical
contact structure on $T^{\ast }Q\times \mathbb{R}$.

This framework naturally addresses situations in which the equations of
motion are not explicitly solvable, such as systems arising from singular
Lagrangians or non-horizontal geometric structures, and extends to
nonholonomic mechanics through implicit Legendrian submanifolds. Typical
applications include dissipative systems and constrained dynamics, where the
Morse family $E$ encodes the singularities of the system, while the contact
structure $\eta ^{T}$ ensures geometric and analytical consistency, even in
cases where standard Hamiltonian formulations break down.\medskip

\noindent \textbullet A significant research direction in geometric
mechanics has focused on generalizing the Hamilton--Jacobi framework to
extend to dynamical systems defined on fibered manifolds. In \cite{GMP,GP1},
Grillo \textit{et al.} develop a generalized Hamilton-Jacobi theory that
unifies classical results from symplectic, Poisson, and almost-Poisson
frameworks. Within this setting, let $\left( M,X\right) $ denote a smooth
phase space equipped with a vector field $X\in \mathfrak{X}(M),$ and
consider a surjective submersion $\pi :\left( M,X\right) \rightarrow N$
where $N$ is a smooth manifold. A section $\sigma :N\rightarrow M$ is said
to be a solution to the $\pi -$Hamilton-Jacobi problem if%
\begin{equation*}
\left. X\right\vert _{\func{Im}\sigma }\in \mathfrak{X}(\func{Im}\sigma ).
\end{equation*}%
This natural formulation generalizes the classical Hamilton-Jacobi equation
by incorporating nonholonomic constraints, time-dependent systems, and
broader geometric structures, while preserving the key connection between
integrability and invariant submanifolds.

The concept of a complete solution is central to this framework.
Specifically, a complete solution $\Sigma :N\times \Lambda \rightarrow M$ is
a surjective local diffeomorphism such that for each $\lambda \in \Lambda ,$
the associated partial solution $\sigma _{\lambda }=\Sigma (\cdot ,\lambda
):N\rightarrow M$ satisfies the $\pi -$Hamilton-Jacobi ensuring the
reconstruction of all trajectories of $X$ from those of the projected vector
fields $\left. X\right\vert _{\func{Im}\sigma _{\lambda }}=\sigma _{\lambda
\ast }\pi _{\ast }\left( \left. X\right\vert _{\func{Im}\sigma _{\lambda
}}\right) .$ This construction also enables the study of integrability by
quadratures in Poisson manifolds, even when the Poisson structure is
degenerate. The key idea is that complete solutions contain both first
integrals and invariant submanifolds, linking geometric mechanics with
explicit solutions.

We now focus on the case where $M$ is a contact manifold equipped with a
contact form $\eta .$ In this setting, the surjective local diffeomorphism $%
\Sigma :N\times \Lambda \rightarrow M$ is a complete solution of the $\pi -$%
Hamilton--Jacobi problem for the contact vector field $X_{H}$ if and only if
the vector field $X_{H}^{\Sigma }$ , the unique $p_{\Lambda }-$vertical
field satisfying $\Sigma _{\ast }(X_{H}^{\Sigma })=\left. X_{H}\right\vert _{%
\func{Im}\Sigma }$, is the Hamiltonian vector field of the function $\Sigma
^{\ast }H$ with respect to the contact form $\Sigma ^{\ast }\eta $ on $%
N\times \Lambda .$

In \cite{GP2} Grillo and Padr\'{o}n establish conditions under which a
complete solution of the $\pi -$Hamilton--Jacobi problem on a contact
manifold $(M,\eta )$ guarantees integrability by quadratures. In this
context, when the solution is pseudo-isotropic, one obtains generating
functions $W_{\lambda }$ satisfying $dW_{\lambda }=\sigma _{\lambda }^{\ast
}\eta ,$ and the identity $L_{X_{H}^{\Sigma }}W=\Sigma ^{\ast }H$ leads to
explicit integration formulas. In the particular case $\mathcal{R}(H)=0,$
the dynamics reduce to linear relations in time for $(\varphi _{\lambda
},W_{\lambda })$ guaranteeing integrability. Extending the framework to
rescaled contact forms $\eta ^{\prime }=g\eta $ with $g=1/H$ covers systems
where $\mathcal{R}(H)\neq 0,$ thus connecting the contact and symplectic
regions $(M_{1},\eta )$ and $(M_{0},d\eta ).$ Consequently, introducing the
notion of bi-isotropic complete solutions yields integration by quadratures
in both domains.\medskip

\noindent \textbullet\ Addressing the Hamilton--Jacobi problem in contact
geometry using homogeneous symplectic tools has been explored through
various proposals in the literature. The work of \cite{LLM} offers a deep
and insightful realization of these ideas, in which the identification of $%
T^{\ast }Q\times \mathbb{R}$ with the projective bundle $\mathcal{P}\left(
T^{\ast }\left( Q\times \mathbb{R}\right) \right) $ is employed. By
excluding the subset at infinity, defined in coordinates $%
(q^{i},z,P_{i},P_{z})$ by $\left\{ P_{z}=0\right\} ,$ one constructs a map%
\begin{equation*}
\Phi :T^{\ast }\left( Q\times \mathbb{R}\right) \backslash \left\{
P_{z}=0\right\} \rightarrow T^{\ast }Q\times \mathbb{R}
\end{equation*}%
which transforms a symplectic Hamiltonian system $(T^{\ast }\left( Q\times 
\mathbb{R}\right) \backslash \left\{ P_{z}=0\right\} ,$\linebreak $\omega
_{Q\times \mathbb{R}},\widetilde{H})$ into the contact Hamiltonian system $%
(T^{\ast }Q\times \mathbb{R},\eta ,H).$ This framework allows for a rigorous
analysis of the relationships between the corresponding Hamiltonian vector
fields, their behavior with respect to sections in both phase spaces
(symplectic and contact), and the interplay of solutions to the
Hamilton-Jacobi problem in both settings.

Another particularly significant perspective, both from the viewpoint of its
theoretical framework and its range of applications, arises from the novel
topological interpretation of contact Hamiltonian functions as sections of a
line bundle introduced by Grabowska \textit{et al.}\ \cite{GG}. One of the
main achievements of this approach is the extension of Hamilton--Jacobi
theory to contact manifolds $(M,C)$, where $C$ denotes a (not necessarily
cooriented) contact structure; however, its implications go beyond this
setting.

By identifying a contact manifold $(M,C)$ with a symplectization-type
structure $(P,\tau,M,\omega)$, where $\tau:P\to M$ is an $\mathbb{R}^{\ast}$%
-principal bundle over $M$ and $\omega$ is a $1$-homogeneous symplectic form
on $P$, contact Hamiltonians can be interpreted, in agreement with earlier
discussions, as sections $\sigma$ of the line bundle $L_{P}^{\ast}$,
equivalently as $1$-homogeneous functions on $P$.

Accordingly, for a Hamiltonian function $H$ on $P$ vanishing on $%
\tau^{-1}(L_{0})$, where $L_{0}\subset M$ is a Legendre submanifold, the
associated Hamilton--Jacobi equation for a section $S:Q\to L^{\ast}$, with $%
L^{\ast}$ the dual of a line bundle $L\to Q$, takes the form 
\begin{equation*}
\sigma \circ j^{1}(S)=0.
\end{equation*}

Here, $j^{1}(S)$ denotes the first jet of the section $S$, and $\sigma $ is
the section of the bundle $L_{P}^{\ast }=J^{1}(L^{\ast })\times _{Q}L^{\ast
} $ associated with the Hamiltonian. This formulation applies to arbitrary
contact manifolds, including non-trivializable examples such as projective
bundles $P(T^{\ast }M)$. In the particular case of extended cotangent
bundles $T^{\ast }Q\times \mathbb{R}$, corresponding to autonomous systems,
it recovers the standard Hamilton--Jacobi equation.

This construction provides a powerful framework for the analysis of contact
Hamiltonian systems and is especially well suited for the development of
Hamilton--Jacobi theory in non-autonomous settings.

In fact, this extension arises from the autonomization of contact structures
and contact Hamiltonians, constituting another significant advance in \cite%
{GG}. For a contact manifold $(M,C)$ with associated symplectic structure $%
(P,\tau ,M,\omega )$, its autonomization is defined by considering the
extended $\mathbb{R}^{\ast }-$principal bundle $(\widetilde{P},\widetilde{%
\tau },\widetilde{M},\widetilde{\omega })$, where $\widetilde{P}=P\times
T^{\ast }\mathbb{R}$, $\widetilde{M}=M\times \mathbb{R}$, and the symplectic
form is given by $\widetilde{\omega }=\omega +dt\wedge dp$, which is again $%
1 $-homogeneous.

A time-dependent contact Hamiltonian is a function $H:P\times \mathbb{R}%
\rightarrow \mathbb{R}$ such that, for each $t\in \mathbb{R}$, the function $%
H_{t}=H(\cdot ,t)$ is a $1$-homogeneous Hamiltonian on $P$. The
autonomization of $H$ is then defined by 
\begin{equation*}
\widetilde{H}(x,t,p)=H(x,t)+p,
\end{equation*}%
which is a $1$-homogeneous Hamiltonian on the extended symplectic $\mathbb{R}%
^{\ast }-$principal bundle $\widetilde{P}$.

A time-dependent section $\sigma:M\times\mathbb{R}\to L_{P}^{\ast}$ is
considered, which can be identified with a section of the line bundle $%
L_{P}^{\ast}\times\mathbb{R}$. Its autonomization is given by 
\begin{equation*}
\widetilde\sigma(v_{y},t)=(v_{y},\,p+\sigma(y,t)),
\end{equation*}
where $\sigma:L_{P}^{\ast}\to V(L_{P}^{\ast})$ represents the time-dependent
Hamiltonian section in the extended framework.

Within this setting, the non-autonomous contact Hamilton--Jacobi equation
takes the form 
\begin{equation*}
\sigma_{0}(j^{1}(S_{t}),t)+\frac{\partial S}{\partial t}(q,t)=0,
\end{equation*}
where $S:Q\times\mathbb{R}\to L^{\ast}$ is a section of the dual bundle, and 
$j^{1}(S_{t})$ denotes the first jet of $S_{t}(q)=S(q,t)$ (see \cite{GG} for
details).

As a particularly relevant special case, consider a Hamiltonian of the form 
\begin{equation*}
H_{0}(z,q_{i},p_{j},t)=H_{1}(q_{i},p_{j})+\lambda t z.
\end{equation*}
This yields a time-dependent version of the discounted Hamilton--Jacobi
equation 
\begin{equation*}
H_{1}\left(q_{i},\frac{\partial S}{\partial q_{j}}\right)+\lambda S=0.
\end{equation*}

This equation captures the essence of discounted dynamics, a physical or
economic phenomenon in which relevant quantities such as energy, utility, or
cost decrease over time, as encoded by the discount term $\lambda S$ in the
Hamilton--Jacobi equation.

This intrinsic structure opens the way to further developments, notably
extensions to singular Lagrangian formalisms through Tulczyjew triples,
linear Jacobi structures, and Lie algebroids.\medskip

\noindent \textbullet\ In \cite{GGU}, Grabowska \textit{et al.}\ introduce
AV-bundles, namely affine fibrations that generalize standard vector
bundles, in order to develop a canonical affine-bracket formulation of
Hamiltonian mechanics. A key result of this approach is the observation that
Jacobi algebroids, or equivalently skew-symmetric algebroids endowed with a
distinguished $1-$cocycle, provide the natural geometric framework for the
description of contact manifolds. More precisely, each Jacobi algebroid
induces a Jacobi structure on its dual bundle, which, under suitable
regularity conditions, recovers a contact structure. Related extensions of
almost contact geometry to the setting of Lie algebroids, including
generalized almost contact structures, have been investigated in \cite{PAB}.

Building upon this geometric framework, the work of Balseiro \textit{et al.}%
\ \cite{BMMP} presents a unified formulation of the Hamilton--Jacobi
equation for a broad class of mechanical systems by means of skew-symmetric
algebroids equipped with a $1-$cocycle. This formalism simultaneously
incorporates dissipation, encoded by the $1-$cocycle $\phi $ which accounts
for non-conservative terms, time dependence, modeled through affine bundles
analogous to the AV-bundles introduced in \cite{GGU}, and nonholonomic
constraints, either linear or affine, which are naturally encoded in the
algebroid structure and its associated differential $d^{E}$.

In contrast with Lie algebroids, whose brackets satisfy the Jacobi identity
and hence induce Poisson structures, skew-symmetric algebroids relax this
requirement, allowing for $(d^{E})^{2}\neq 0$. This weakening of the Jacobi
identity provides a natural geometric framework for modeling systems with
external forces or non-conservative dynamics. In this setting, the presence
of a distinguished $1-$cocycle $\phi $ plays the role of a underlying
geometric structure, unifying and generalizing several familiar examples:
the electromagnetic potential in gauge-theoretic models, the contact $1-$%
form in contact geometry, and friction or damping effects arising in
non-classical mechanical systems.

We now turn to the geometric framework that provides a natural setting for
all the aforementioned structures. Specifically, let $(E,Q)$ be a
skew-symmetric algebroid of rank $n$ over the manifold $Q$, endowed with an
anchor map $\rho :E\rightarrow TQ$ and an $\mathbb{R}$-bilinear
skew-symmetric bracket on $\Gamma (E)$ satisfying the Leibniz rule 
\begin{equation*}
\lbrack \sigma ,f\gamma ]=f[\sigma ,\gamma ]+\rho (\sigma )(f)\gamma ,\qquad
\sigma ,\gamma \in \Gamma (E),\;f\in C^{\infty }(Q).
\end{equation*}%
This structure generalizes the Lie algebroid framework by relaxing the
Jacobi identity, as discussed in \cite{GGU} in the context of affine
geometries.

Let $\phi\in\Gamma(E^{\ast})$ be a $1$-cocycle, that is, a section of the
dual bundle satisfying $d^{E}\phi=0$, and consider the vector subbundle $%
V=\phi^{-1}(0)\hookrightarrow E$. The inclusion of $V$ induces on it a
skew-symmetric algebroid structure. Moreover, the projection 
\begin{equation*}
\mu:E^{\ast}\to V^{\ast}
\end{equation*}
defines an affine bundle of rank $1$, modeled over the trivial vector bundle 
$V^{\ast}\times\mathbb{R}\to V^{\ast}$. A section $h\in\Gamma(\mu)$ is
called a Hamiltonian section, and the quadruple $(E,\rho,\phi,h)$ defines a
Hamiltonian system on the skew-symmetric algebroid $E$.

As shown in \cite{BMMP}, the Hamiltonian section $h$ induces a Hamiltonian
vector field $R_{h}$ on $V^{\ast}$, whose integral curves correspond to
solutions of the Hamilton equations in this generalized framework.

The generalized Hamilton--Jacobi equation for a Hamiltonian system on a
skew-symmetric algebroid equipped with a $1$-cocycle is given by 
\begin{equation*}
\mu\circ\bigl(\zeta_{h}^{\alpha}\, d^{E}(h\circ\alpha)\bigr)=0,
\end{equation*}
where $\zeta_{h}^{\alpha}$ is a section of $E$ constructed from a section $%
\alpha\in\Gamma(V^{\ast})$ via the affine structure of the bundle $%
\mu:E^{\ast}\to V^{\ast}$ (see \cite{BMMP} for details).

This equation unifies the classical and contact versions of Hamilton--Jacobi
theory. The presence of the $1$-cocycle $\phi $ naturally introduces
dissipative effects, such as external forces or friction, while the affine
bundle structure of $\mu $ generalizes contact geometry. In this way, both
the classical Hamiltonian framework on symplectic manifolds and the contact
framework based on Legendrian submanifolds are recovered as particular
cases, through suitable reduction procedures and affine Poisson morphisms,
as discussed in \cite{GGU} and further developed in \cite{BMMP}.

To conclude, we emphasize that the mathematical structures discussed here,
Lie and Jacobi algebroids together with their skew-symmetric extensions
equipped with a $1$-cocycle, are, in our view, poised to play a central role
in the future development of the geometric frameworks considered in this
work and their applications. Their combination of affine, cohomological, and
variational features, along with their ability to unify dissipative systems,
nonholonomic constraints, and time-dependent dynamics, suggests that they
may provide a natural setting for addressing certain aspects of quantum
field theory, including the quantization of dissipative systems, open
quantum systems, and non-equilibrium field theories.\medskip

\noindent \textbullet\ After discussing in some detail a selection of the
more representative approaches to Hamilton--Jacobi theory, we conclude by
briefly collecting other formulations across a wide spectrum of related
research lines. A unified overview of Hamilton--Jacobi theory in the
symplectic, cosymplectic, and contact settings is presented by Rom\'{a}n-Roy
in \cite{roman}. A viewpoint emphasizing the role of generalized algebraic
structures (Jacobi and Leibniz) is presented by Esen \textit{et al.} \cite%
{ELLSZ}. In \cite{LLLa}, de Le\'{o}n \textit{et al.} address integrability
of contact Hamiltonian systems, action--angle coordinates, geometric
reduction, and applications to dissipative systems. Along these lines,
Azuaje \cite{azuaje} establishes a Lie integrability-by-quadratures theorem
for Hamiltonian systems in the symplectic, cosymplectic, contact, and
cocontact settings, showing that the existence of a solvable Lie algebra of
symmetries guarantees integrability by quadratures. For certain
applications, Ni \textit{et al.} \cite{NWY} study the contact
Hamilton--Jacobi equation from the viewpoint of viscosity solutions, with
applications to the long-time behavior of dissipative Hamiltonian systems.
Cannarsa \textit{et al.} \cite{CCWY} study contact-type Hamilton--Jacobi
equations through Herglotz generalized variational principle, establishing
connections with weak KAM theory and the Lax--Oleinik evolution in
nonconservative dynamics. Wang \textit{et al}. \cite{WWY-b} solve contact
Hamilton--Jacobi equations on compact manifolds and establishes connections
with weak KAM theory and ergodic properties in nonconservative dynamics.
Finally, Esen \textit{et al.} \cite{ESZ} starting from a discrete contact
Lagrangian formulation, develop a Hamilton--Jacobi theory for continuous
contact Hamiltonian dynamics, with discrete Hamiltonian equations obtained
through a discrete Legendre transformation.

\section{Contact momentum maps}

Symplectic reduction is a fundamental construction in geometry and
Hamiltonian mechanics for the study of systems with symmetries. By
eliminating the degrees of freedom associated with these symmetries, it
produces a reduced Hamiltonian system on a lower-dimensional phase space.
This framework plays a central role in classical mechanics and in several
areas of theoretical physics, including gauge theories, fluid mechanics, and
general relativity.

In their seminal article \cite{MW}, Marsden and Weinstein unified several
classical reduction techniques into a unified geometric framework,
profoundly influencing geometric mechanics. Their construction, based on the
moment map $J:P\rightarrow \mathfrak{g}^{\ast }$ and the induced symplectic
structure on the reduced phase space $P_{\mu }=J^{-1}(\mu )/G_{\mu }$
provided a rigorous geometric language for Hamiltonian systems with
symmetries. The resulting reduced phase space preserves the essential
dynamical content of the original system while removing geometric
redundancies.

However, the classical theory of symplectic reduction proved insufficient
for the treatment of Hamiltonian systems subject to nonholonomic
constraints. To overcome this limitation, Bates and \'{S}niatycki \cite{BS}
developed a framework for symmetry reduction on constraint manifolds,
highlighting the role played by symplectic structures adapted to the
constrained dynamics. Their results were restricted to linear constraints,
where the associated constraint distribution forms a vector subbundle of the
tangent bundle. The extension of these ideas to nonlinear nonholonomic
constraints motivated further developments, as explored in subsequent works
such as \cite{Sn} and \cite{CLMD,LMD}.

In the presence of symmetries in dissipative Hamiltonian systems, contact
geometry naturally replaces the symplectic framework. The problem of
reduction then consists in constructing lower-dimensional contact manifolds
that faithfully encode the constrained dynamics. In this context,
coisotropic reduction provides a fundamental mechanism for the analysis of
dissipative systems with constraints. The reduced contact space supports the
dynamics in terms of Legendrian submanifolds, which represent admissible
physical solutions.

These ideas have been developed in works such as \cite{dLL,Lo,Wi}. In this
section, we revisit these constructions by showing how the symplectic cover
channels the contact momentum map, thereby providing deeper insight into the
corresponding reduction procedures.

Let $G$ be a Lie group with Lie algebra $\mathfrak{g}$ and dual $\mathfrak{g}%
^{\ast }.$ We assume that $G$ acts by contactomorphism on the contact
manifold $(M_{c},\alpha ).$ Denote by $(M^{s},\Omega =d\omega )$ the
symplectic cover of $M_{c}$ as introduced in Section 3$.$ The action of $G$
lifts naturally to $M^{s}$ as follows%
\begin{equation*}
g:M^{s}\rightarrow M^{s}:(p,t\alpha _{p})\mapsto \left( g(p),t\alpha
_{p}\right) ,\quad t>0.
\end{equation*}

Given $\xi \in \mathfrak{g,}$ let $X_{\xi }$ denote the fundamental vector
field on $M^{s}$ associated with the lifted action of $G$. Since the lifted
action preserves the $1-$form $\omega $, as given by the local expression
(1), it follows that%
\begin{equation}
\mathcal{L}_{X_{\xi }}\omega =0,\text{\quad }\forall \xi \in \mathfrak{g}. 
\tag{13}
\end{equation}

As a consequence of (13), we obtain%
\begin{equation*}
i_{X_{\xi }^{{}}}\Omega =-d\left\{ \omega (X_{\xi }^{{}})\right\} ,\text{%
\quad }\forall \xi \in \mathfrak{g.}
\end{equation*}%
This allows us to define the moment map $J:M^{s}\rightarrow \mathfrak{g}%
^{\ast }$ by%
\begin{equation}
J(q)(\xi )=w_{q}(X_{\xi }^{{}})=t\alpha _{p}(X_{\xi }^{{}}),\text{\quad }%
q=(p,t\alpha _{p})\in M^{s}.  \tag{14}
\end{equation}%
It is well known that this momentum map is equivariant with respect to the
coadjoint action of $G,$ namely $J(g\cdot q)=Ad_{g^{-1}}^{\ast }J(q),$ for
all $q\in M^{s}$ and $g\in G.$ See, for instance, Theorem 4.2.10 in \cite{AM}%
.

Relation (14) motivates the following definition.

\begin{definition}
Let $(M_{c},\alpha )$ be a contact manifold and let $G$ be a Lie group
acting on $M_{c}$ by a contactomorphism. The contact momentum map is the map%
\begin{equation*}
j:M_{c}\rightarrow \mathfrak{g}^{\ast }
\end{equation*}%
defined by%
\begin{equation*}
j(p)(\xi ):=\alpha _{p}(X_{\xi }^{{}}),\text{\quad }\xi \in \mathfrak{g},
\end{equation*}%
where $X_{\xi }^{{}}$ denotes the fundamental vector field on $M_{c}$
associated with $\xi .$
\end{definition}

By abuse of notation, we denote by $X_{\xi }^{{}}$ both the fundamental
vector field on $M_{c}$ and its lift to $M^{s},$ which projects onto $X_{\xi
}^{{}}$ under the canonical projection $\pi :M^{s}\rightarrow M_{c}.$

\begin{lemma}
With the above notation, we have 
\begin{equation*}
\pi ^{-1}\left[ j^{-1}(0)\right] =J^{-1}(0).
\end{equation*}
\end{lemma}

\begin{proposition}
If $0\in \mathfrak{g}^{\ast }$ is a regular value of the moment map $%
j:(M_{c},\alpha )\rightarrow \mathfrak{g}^{\ast }$, then $j^{-1}(0)$ is a
coisotropic submanifold of the contact manifold $M_{c}$.
\end{proposition}

\begin{proof}
By definition of coisotropic submanifolds in contact geometry (see Section
3), it suffices to show that $\pi ^{-1}(j^{-1}(0))=J^{-1}(0)\subset M^{s}$
is coisotropic with respect to the symplectic form $\Omega .$ Since $0$ is a
regular value of $j,$ it is also a regular value of $J$.

For every $q\in J^{-1}(0)$, one has 
\begin{equation*}
T_{q}(G\cdot q)=T_{q}(J^{-1}(0))^{\perp _{\Omega }},
\end{equation*}%
where $\perp _{\Omega }$ denotes the symplectic orthogonal (see, Lemma~4.3.2
in \cite{AM}). Hence, 
\begin{equation*}
T_{q}(J^{-1}(0))^{\perp _{\Omega }}=T_{q}(G\cdot q)\subset T_{q}(J^{-1}(0)),
\end{equation*}%
which shows that $J^{-1}(0)$ is a coisotropic submanifold of $M^{s}.$
\end{proof}

\begin{proposition}
Let $(M,\alpha )$ be a contact manifold on which a Lie group $G$ acts by
contactomorphisms. Assume that $0\in \mathfrak{g}^{\ast }$ is a regular
value of the contact momentum map $j$, and that the action of $G$ on $%
j^{-1}(0)$ is free and proper. Then the quotient manifold 
\begin{equation*}
j^{-1}(0)/G
\end{equation*}
carries a unique contact structure whose contact form is induced by $\alpha $%
.
\end{proposition}

\begin{proof}
Let $p\in j^{-1}(0)$. Since $\alpha _{p}(X_{\xi })=0$ for all $\xi \in 
\mathfrak{g}$, the contact form $\alpha $ vanishes along the tangent spaces
to the $G$--orbits. Hence $\alpha $ induces a well-defined one-form on the
quotient space $T_{p}(j^{-1}(0))/T_{p}(G\cdot p)$.

Because $0$ is a regular value of $j$, the submanifold $j^{-1}(0)\subset M$
has codimension $\dim G$. Moreover, the fundamental vector fields $X_{\xi }$
span the characteristic distribution of $j^{-1}(0)$, which coincides with
the kernel of the restriction of $d\alpha $ to $\ker \alpha \cap
T(j^{-1}(0)) $. This follows from the relation 
\begin{equation*}
i_{X_{\xi }}d\omega =0\text{ on }J^{-1}(0),\qquad \forall \xi \in \mathfrak{g%
},
\end{equation*}%
whose restriction to the hypersurface $t=1$ yields $\iota _{X_{\xi }}d\alpha
=0$ on $j^{-1}(0).$ Therefore, the $2-$form $d\alpha $ is nondegenerate on
the quotient $(\ker \alpha \cap T(j^{-1}(0)))/G$, which shows that the
induced $1-$form defines a contact structure on $j^{-1}(0)/G$. The
uniqueness of this contact structure follows from the fact that it is
completely determined by the restriction of $\alpha $.\medskip
\end{proof}

\noindent It should be noted that the proof of Proposition 23 shows that,
within our framework, Theorem 19 in \cite{dLL} concerning reduction via
contact momentum maps is recovered only in the case $\mu =0.$ Indeed, the
key insight in \cite{dLL} is that the characteristic distribution of a
coisotropic submanifold $N,$ defined by%
\begin{equation*}
TN^{\perp _{\wedge }}=\ker \left( \left. \eta \right\vert _{TN}\right) \cap
\ker \left( \left. d\eta \right\vert _{TN}\right)
\end{equation*}%
which is the contact analogue of the symplectic orthogonal, integrates to a
foliation of the constraint manifold. Taking the quotient by this
characteristic foliation yields a new contact manifold of lower dimension.
As a consequence, when $\mu $ is a regular value of the moment map $j$ the
reduced space 
\begin{equation*}
j^{-1}(\mu )/T(j^{-1}(\mu ))^{\perp _{\wedge }}
\end{equation*}%
inherits a natural contact structure. However, in general, this quotient
does not coincide with the orbit space $j^{-1}(\mu )/G.$

The foliation determined by the characteristic distribution also suggests
natural extensions of contact reduction to regular values $%
\mu
\neq 0,$ as in the geometric construction described by Willett in \cite{Wi}.
More precisely, for $\mu \in \mathfrak{g}^{\ast }$, let $K_{\mu }$ denote
the connected Lie subgroup of $G_{\mu }$ with Lie algebra $\mathfrak{k}_{\mu
}=\ker (\mu |_{\mathfrak{g}_{\mu }})$. If $K_{\mu }$ acts properly on $\Phi
_{\alpha }^{-1}(\mathbb{R}^{+}\mu )$, $\Phi _{\alpha }$ is transverse to $%
\mathbb{R}^{+}\mu $ and $\ker \mu +\mathfrak{g}_{\mu }=\mathfrak{g}$, then
the quotient 
\begin{equation*}
M_{\mu }:=\Phi _{\alpha }^{-1}(\mathbb{R}^{+}\mu )/K_{\mu }
\end{equation*}%
carries a natural contact orbifold structure endowed with an induced contact
form $\alpha _{\mu }$ (see details in \cite{Wi}). This construction
generalizes the special case $\mu =0$, previously studied in \cite{Al,Lo},
and provides a unified framework for contact reduction at arbitrary regular
values of the momentum map.

Furthermore, \cite{Wi} introduces a stratified reduction scheme that applies
when the transversality condition is relaxed. In this setting,
transversality is replaced by the existence of a convex, $\mathbb{R}^{+}-$%
invariant slice associated with $\mu \in \mathfrak{g}^{\ast }$. Under this
assumption, the quotient $M_{\mu }$ is realized as a stratified space: 
\begin{equation*}
M_{\mu }=\bigsqcup_{(H)}(M_{\mu })_{(H)},
\end{equation*}%
where the strata $(M_{\mu })_{(H)}$, indexed by the conjugacy classes of
stabilizer subgroups of $K_{\mu }$, each inherit a cooriented contact
structure. This construction extends the corresponding result for $\mu =0$
developed by Lerman and Willet in \cite{LW}, and yields a comprehensive
geometric description of contact reduction in the presence of singularities
and non-free group actions.

\medskip

\noindent \textbullet\ We conclude this section with a brief reference to
other contributions on contact momentum maps and reduction. Closely related
to the present work, Garc\'{\i}a-Mauri\~{n}o \cite{garcia} studies the
symplectification of contact manifolds and its commutativity with
coisotropic and momentum map reduction, and analyzes the behavior of the
Reeb vector field under contact reduction.

Boyer and Galicki \cite{boyer-galicki} introduced the notion of toric
contact manifolds and analyzed contact momentum maps for torus actions
preserving a contact form. Their work establishes foundational results on
contact toric geometry. Building on this foundational work Lerman \cite%
{lerman} developed a systematic framework for contact toric manifolds,
introducing the notion of the contact moment cone and clarifying the
structure of torus actions in contact geometry.

Munteanu \textit{et al.} \cite{munteanu-rey-salgado} developed a notion of
momentum map and a corresponding reduction procedure within G\"{u}nther's
formalism (a covariant geometric framework generalizing symplectic mechanics
to field-theoretic settings), extending symmetry reduction beyond the
standard symplectic setting. In the $k-$contact setting, de Lucas \textit{et
al.} \cite{deLucas-Rivas-Vilarino-Zawora} develop a Marsden--Weinstein
reduction scheme and analyze how the $k-$contact Hamiltonian dynamics
projects onto the reduced space. De Le\'{o}n and Izquierdo-L\'{o}pez \cite%
{deLeon-Izquierdo} provide a comprehensive review of coisotropic reduction
in symplectic, cosymplectic, contact, and co-contact Hamiltonian systems,
emphasizing a unified geometric framework compatible with the underlying
dynamics.

Colombo \textit{et al.} \cite{colombo-deleon-eyrea-lopez} develop a
geometric formulation of combined conservative--dissipative mechanical
systems within contact Hamiltonian dynamics, analyzing symmetries and
reduction, and proposing variational integrators that preserve the
underlying geometric structure. Xia and Xin \cite{xia-xin} study singular
reduction of contact Hamiltonian systems, extending contact reduction to
non-regular group actions and analyzing the geometric structure and induced
dynamics on the resulting stratified orbit spaces. Do and Oh \cite{do-oh}
introduce a thermodynamic reduction procedure for contact dynamics, deriving
effective macroscopic equations from contact Hamiltonian systems and
providing a geometric framework for nonequilibrium thermodynamic dynamics.

\section{Symmetries and Dissipation}

The geometric understanding of dissipative processes has gained significant
attention in recent years \cite{Raz}. Within this context, contact geometry
has emerged as a natural framework for extending classical Hamiltonian
mechanics beyond the conservative setting. Although contact structures
appear in mechanics through the Poincar\'{e}--Cartan $1-$form associated
with time-dependent Hamiltonians, a decisive conceptual step was taken in 
\cite{BCT}. There, the phase space of mechanical systems, both conservative
and dissipative, is reinterpreted as a contact manifold, providing a genuine
extension of the traditional symplectic formalism. The key innovation
consists in the introduction of an additional dimension, encoded by a
dynamical variable $z$ which naturally accounts for the interaction of the
system with its environment.

This perspective has evolved into a robust geometric framework for modeling
mechanical systems subject to dissipation, as illustrated in \cite%
{Brr,Br,BLN,LL,GLR,Vi}. In particular, contact Hamiltonian dynamics extends
the classical Hamiltonian equations in a way that intrinsically incorporates
dissipative effects, yielding a natural generalization of symplectic
dynamics to non-conservative systems.

A fundamental distinction between symplectic and contact Hamiltonian
dynamics lies in the non-conservation of the Hamiltonian $H$. Along
solutions of the contact Hamiltonian equations, one has \ \ 
\begin{equation*}
\frac{d}{dt}H=H\frac{\partial H}{\partial z},
\end{equation*}%
which directly reflects the intrinsically dissipative character of contact
systems.

Just as the analysis of conserved quantities in autonomous Hamiltonian
systems relies on a well-established geometric framework (see, e.g., \cite%
{RR}), the analysis of dissipative quantities in contact systems has
likewise been shown to be deeply rooted in the geometric structure of the
extended phase space. Seminal works such as \cite{GGMRRx} and \cite{LL}
provided the foundational insights for this field.

The paradigmatic example of classical conservation laws is provided by
Noether's theorem. Originally formulated in the context of Lagrangian and
Hamiltonian mechanics, modern formulations of the theorem are grounded in a
geometric framework for a Lagrangian system $L:TQ\rightarrow \mathbb{R}$
(see, e.g., \cite{crampin1983}). In this setting, the invariance condition $%
X^{c}(L)=0$, where $X^{c}$ denotes the complete lift to $TQ$ of a vector
field $X$ on $Q$, implies that $X^{v}(L)$ is a conserved quantity. Here, $%
X^{v}$ denotes the vertical lift of $X$ to $TQ$, canonically associated with 
$X$.

This principle has been successfully extended and reformulated across a wide
range of contexts, including field theory, general relativity, and quantum
physics, highlighting the simplicity and universality of one of the most
fundamental principles in theoretical physic (see, e.g., \cite{kosmann}).

A striking result due to de Le\'{o}n and Lainz \cite{LL} shows that, in the
contact Lagrangian setting $L:TQ\times \mathbb{R}\rightarrow \mathbb{R}$,
the quantity $X^{v}(L)$ dissipates in the same way as the energy of the
system, namely $H=v^{i}\partial L/\partial v^{i}-L$. This result provides a
natural starting point for the development of a theory of dissipative
mechanical systems, a field that is still in its early stages.

To begin, let $(M,\eta )$ be a $(2n+1)-$dimensional contact manifold endowed
with a contact $1-$form $\eta ,$ so that, $\eta \wedge (d\eta )^{n}$ is a
volume form on $M.$ The Reeb vector field $\mathcal{R}$, satisfies%
\begin{equation*}
i_{\mathcal{R}}\left( \eta \wedge (d\eta )^{n}\right) =(d\eta )^{n}.
\end{equation*}%
As recalled in Section 2, to each smooth function $H$ defined on $M,$ one
associates a unique Hamiltonian vector field $X_{H}\in \mathfrak{X}(M)$,
satisfying%
\begin{align*}
\eta (X_{H})& =-H \\
\mathcal{L}_{X_{H}}\eta & =-\mathcal{R}(H)\eta .
\end{align*}

For a fixed Hamiltonian function $H\in C^{\infty }(M),$ the triple $(M,\eta
,H)$ will be referred to as a contact Hamiltonian system.

We also recall from Section 2 that, for any pair of smooth functions $f,g\in
C^{\infty }(M),$ the contact bracket is defined by%
\begin{equation*}
\{f,g\}=-i_{[X_{f},X_{g}]}\eta .
\end{equation*}%
(a different sign convention was adopted in Section 2 for convenience). This
can be equivalently expressed as%
\begin{equation*}
X_{f}(g)+g\mathcal{R}(f)
\end{equation*}%
where $X_{f}$ and $X_{g}$ denote the Hamiltonian vector fields associated
with $f$ and $g$, respectively. A notable special case arises when
considering the action of $X_{H}$ on $H$ itself, yielding the dissipation
law for the Hamiltonian $H$ 
\begin{equation*}
X_{H}(H)=-\mathcal{R}(H)\cdot H.
\end{equation*}%
In this context, a real-valued function $f$ on the contact Hamiltonian
system $(M,\eta ,H)$ is called a dissipated quantity if it satisfies $%
\{H,f\}=0$, or equivalently, $X_{H}(f)=-\mathcal{R}(H)\cdot f.$ Analogously,
a function $g\in C^{\infty }(M)$ is said to be a conserved quantity whenever 
$X_{H}(g)=0.$ It then follows directly that if $f_{1}$ and $f_{2}$ are
dissipated quantities, then the ratio $f_{1}/f_{2}$ defines a conserved
quantity on the open set where $f_{2}\neq 0.$

The search for dissipation laws associated with symmetries of a contact
Hamiltonian system begins with the following definition.

\begin{definition}
A vector field $Y\in \mathfrak{X(}M)$ is called a dynamical symmetry of the
contact Hamiltonian system $(M,\eta ,H)$ if $\mathcal{L}_{Y}X_{H}=0.$
\end{definition}

With this definition, it is straightforward to show that if $Y$ is a
dynamical symmetry, then $\eta (Y)$ is a dissipated quantity. Indeed, this
follows from the key relation established in \cite{LL},%
\begin{equation*}
\{H,\eta (Y)\}=\eta ([X_{H},Y]).
\end{equation*}

On the other hand, let $X$ be a dynamical symmetry of the contact
Hamiltonian system $(M,\eta ,H)$ such that the Hamiltonian $H$ is constant
along the flow of $X$ and assume that $\mathcal{L}_{X}\eta =dg,$ for some
smooth function $g\in C^{\infty }(M).$ Then $g$ is a conserved quantity.

This result follows from the remarkable identity established in \cite{LL},%
\begin{equation*}
\{H,f\}=\left( \mathcal{L}_{X}\eta \right) (X_{H})+X(H).
\end{equation*}

When a Lie group acts by preserving the underlying geometric structures,
characteristic quantities naturally arise, as asserted by Noether's theorem.
In the contact setting, such quantities typically appear as dissipated
quantities, a fact captured by the following fundamental result. This
statement was obtained in \cite{LL} as an application of the contact
momentum map, and independently in \cite{Pe} through the analysis of the
invariance of the orbits of the Hamiltonian vector field.

\begin{proposition}
Let $G$ be a Lie group acting by contactomorphisms on the contact
Hamiltonian system $(M,\eta ,H),$ and assume that the Hamiltonian function $%
H $ is $G-$invariant, i.e.,$\ g^{\ast }H=H,$ for all $g\in G.$ Then, for any 
$\xi \in \mathfrak{g,}$ the Hamiltonian function $f_{X_{\xi }}$ associated
with the fundamental vector field $X_{\xi }\in \mathfrak{X}(M)$ is a
dissipated quantity, 
\begin{equation*}
\{H,f_{X_{\xi }}\}=0.
\end{equation*}
\end{proposition}

Lagrangian mechanics is founded on a well-established geometric framework
defined on the tangent bundle. In particular, the almost tangent structure
of the tangent bundle $TM$ of a manifold $M$, described by $(1,1)-$tensor
field, provides the geometric backbone of the Lagrangian formalism (see,
e.g., \cite{SCC} for foundational aspects of this theory).

This geometric framework extends naturally to contact Lagrangian systems,
where the tangent bundle is replaced by an extended phase space endowed with
a contact structure. This setting enables the geometric description of
non-conservative dynamics. In this context, the Lagrangian formalism is
enriched by the interplay between the almost tangent structure and contact
geometry, giving rise to dissipated quantities as natural counterparts of
conserved quantities in symplectic mechanics. This approach generalizes
classical Lagrangian mechanics and also provides a clear geometric
interpretation of dissipation, as developed in \cite{GGMRRx}.

To make this geometric structure explicit in local coordinates, le $Q$ be an 
$n-$dimensional manifold, and denote by $TQ$ its tangent bundle. The
extended phase space is given by $TQ\times \mathbb{R,}$ endowed with local
coordinates $(q^{i},v^{j},z)$. The canonical almost tangent structure $J$ on 
$TQ,$ locally expressed as%
\begin{equation*}
J=\frac{\partial }{\partial v^{i}}\otimes dq^{i}
\end{equation*}%
admits a natural extension to $TQ\times \mathbb{R.}$ In addition, the
Liouville vector field $\Delta $ on $TQ\times \mathbb{R}$ is defined locally
by $\Delta =v^{i}\partial /\partial v^{i}.$ Let $L:TQ\times \mathbb{%
R\rightarrow R,}$ be a regular Lagrangian function, as recalled in Section
4, i.e., its Hessian matrix with respect to the velocities is non-degenerate
at every point. As a consequence, the $1-$form on $TQ\times \mathbb{R}$
defined by%
\begin{equation*}
\eta _{L}:=dz-dL\circ J
\end{equation*}%
is a contact form. In local coordinates, we introduce the notation 
\begin{equation*}
\alpha _{L}:=dL\circ J=\frac{\partial L}{\partial v^{i}}dq^{i}.
\end{equation*}%
The pair $(TQ\times \mathbb{R},\eta _{L})$ together with the Lagrangian $L$,
is referred to as a contact Lagrangian system.

The energy function associated with the system is defined as%
\begin{equation*}
E_{L}=\Delta L-L=v^{i}\frac{\partial L}{\partial v^{i}}-L.
\end{equation*}

The Hamiltonian vector field associated with the energy function $E_{L}$
will be denoted by $\xi _{L}:=X_{E_{L}}$. Thus, the triple $(TQ\times 
\mathbb{R},\eta _{L},E_{L})$ defines a contact Hamiltonian system naturally
associated with the contact Lagrangian structure. Within this framework, one
of the seminal results of the theory asserts that the characteristics
associated with Noether symmetries are dissipated rather than conserved in
the contact setting \cite{LL}. Indeed, let $Y\in \mathfrak{X}(Q).$ The
complete lift $Y^{c}$ to $TQ\times \mathbb{R}$ is an infinitesimal symmetry
of the Lagrangian $L,$ that is, $Y^{c}(L)=0$ if and only if $Y^{v}(L)$ is a
dissipated quantity, where $Y^{v}=J(Y^{c})$ denotes the vertical lift of $Y.$
This dissipative condition is expressed by the commutativity condition :%
\begin{equation*}
\left\{ E_{L},Y^{v}(L)\right\} =0
\end{equation*}

Let $\mathfrak{X}_{proj}(TQ\times \mathbb{R})\mathbb{\ }$denote the set of
vector fields on $TQ\times \mathbb{R}$ that are projectable onto both $Q$
and $\mathbb{R}$. In local coordinates $(q^{j},v^{i},z)$ any such vector
field can be written as

\begin{equation*}
X=F_{i}(q)\frac{\partial }{\partial q^{i}}+G_{i}(q,v,z)\frac{\partial }{%
\partial v^{i}}+H(z)\frac{\partial }{\partial z},\text{\quad }(1\leq i\leq
n).
\end{equation*}

The following definition extends both the notion of an infinitesimal
symmetry of the Lagrangian arising from complete lifts of vector fields on $%
Q $ and the associated dissipative property of the vertical lift $X^{v}(L)$ 
\cite{LL}.

\begin{definition}
A vector field $X\in \mathfrak{X}_{proj}(TQ\times \mathbb{R})$ is called an
infinitesimal symmetry of $L$ if 
\begin{equation*}
X(L)=-\mathcal{R}_{L}(f)\cdot L,
\end{equation*}%
where $f=-\eta _{L}(X)$ and $\mathcal{R}_{L}$ denotes the Reeb vector field
associated with the contact form $\eta _{L}.$
\end{definition}

The subsequent result \cite{Pe}, delineates the geometric premises that
underpin this novel dissipative object. In particular, it formalizes the
conditions under which such structures emerge, thereby bridging the gap
between symmetry properties and their dissipative manifestations in the
contact Lagrangian setting.

\begin{proposition}
Let $X\in \mathfrak{X}(TQ\times \mathbb{R)}$ be a vector field preserving
the almost tangent structure $J$ invariant, i.e., $\mathcal{L}_{X}J=0.$ Then
the following relation holds, $\mathcal{L}_{X}\alpha _{L}=\alpha _{X(L)}$.
Moreover, if\ $X\in \mathfrak{X}_{proj}(TQ\times \mathbb{R})$ is an
infinitesimal symmetry of the Lagrangian $L$, then $X$ is a contact
Hamiltonian vector field with respect to the contact structure defined by $%
\eta _{L}.$
\end{proposition}

In particular, writing $\eta _{L}(X)=-f$ and $X=X_{f},$ one obtains 
\begin{equation*}
\mathcal{L}_{X_{f}}\eta _{L}=-\mathcal{R}_{L}(f)\eta _{L}.
\end{equation*}%
This result, which is of independent interest, constitutes a fundamental
ingredient for the central theorem that follows.

\begin{theorem}
With the previous notation and assumptions, the function $f=-\eta _{L}(X)$
is a dissipated quantity.
\end{theorem}

Indeed, the following identity holds \cite{LL}%
\begin{equation*}
\{E_{L},f\}=\{L,f\}.
\end{equation*}%
Expanding the right-hand side, one obtains 
\begin{equation*}
\{L,f\}=X_{f}(L)+\mathcal{R}_{L}(f)L=-\mathcal{R}_{L}(f)\cdot L+\mathcal{R}%
_{L}(f)\cdot L=0.
\end{equation*}

Extending the results of Proposition 25, we now study dissipation laws
associated with elements of the Lie algebra $\mathfrak{g}$ of a Lie group $G$
acting on a configuration manifold $Q.$ We assume that this action lifts
naturally to an action on the extended phase space $TQ\times \mathbb{R.}$ In
this setting, we analyze how dissipation laws govern the behavior of
infinitesimal generators of the group action in a non-conservative context,
thereby linking symmetry properties with their corresponding dissipative
dynamics.

More precisely, given a contact Lagrangian system determined by a regular
Lagrangian $L:TQ\times \mathbb{R}\rightarrow \mathbb{R}$, our goal is to
examine the behavior of the associated contact Lagrangian form $\eta
_{L}=dz-\alpha _{L},$ under the induced action of the group.

\begin{definition}
Let $\varphi $ be a diffeomorphism of $Q$ and $L\in C^{\infty }(TQ\times 
\mathbb{R}).$ Denote by $\overline{\varphi }:TQ\times \mathbb{R\rightarrow }%
TQ\times \mathbb{R}$ the product of the tangent lift of $\varphi $ and the
identify on $\mathbb{R}.$ We say that $\varphi $ is a gauge symmetry of the
contact Lagrangian system defined by $L,$ if $\overline{\varphi }^{\ast
}L=L+\alpha ,$ where $\alpha $ is the pullback of a closed $1-$form on $Q$.
\end{definition}

We adopt the following notation. Given a gauge symmetry $\varphi ,$ we write%
\begin{equation*}
\widetilde{L}:=\overline{\varphi }^{\ast }L=L+\alpha \equiv L+f_{\alpha }
\end{equation*}%
where $f_{\alpha }:TQ\rightarrow \mathbb{R}$ denotes the function induced by
the $1-$form $\alpha $ on $Q$. Accordingly, the associated contact
Lagrangian $1-$form transforms as%
\begin{equation*}
\eta _{\widetilde{L}}=dz-i_{J}d\widetilde{L}=dz-i_{J}dL-f_{\alpha }\equiv
dz-i_{J}dL-\alpha .
\end{equation*}%
Thus, we obtain two contact structures on $TQ\times \mathbb{R,}$ determined
by the contact forms $\eta _{L}$ and $\eta _{\widetilde{L}}.$ Despite this
difference, their associated Lagrangian energies agree%
\begin{equation*}
E_{\widetilde{L}}=\Delta (\widetilde{L})-\widetilde{L}=\Delta (L)+f_{\alpha
}-(L+f_{\alpha })=E_{L}.
\end{equation*}%
Moreover, the corresponding Reeb vector fields are the same, 
\begin{equation*}
\mathcal{R}_{L}=\mathcal{R}\widetilde{_{L}},
\end{equation*}%
and, since $\alpha $ is closed, the exterior derivatives of the contact
forms also agree,%
\begin{equation*}
d\eta _{L}=d\eta _{\widetilde{L}}.
\end{equation*}%
The following proposition clarifies how contact Lagrangian systems with
gauge symmetries are incorporated into the geometric framework developed
above.

\begin{proposition}
Let $G$ be a Lie group of transformations on $Q$ such that its natural lift
to $TQ\times \mathbb{R}$ satisfies 
\begin{equation*}
g^{\ast }L=L+\alpha ,\text{\quad }g\in G,
\end{equation*}%
where $\alpha $ is a closed $1-$form on $Q.$ Assume moreover that the
induced action of $G$ preserves the Hamiltonian vector field $\xi _{L}$ of
the contact Lagrangian system $(TQ\times \mathbb{R},\eta _{L},E_{L})$, that
is, 
\begin{equation*}
g_{\ast }\xi _{L}=\xi _{L},\quad \forall g\in G.
\end{equation*}%
Then, for each $\xi \in \mathfrak{g}$, the function $\eta _{L}(X_{\xi })$ is
a dissipated quantity on $TQ\times \mathbb{R}$. In addition, if $H^{1}(Q)=0$%
, there exists a function $h_{\xi }$ on $Q$ such that $\mathcal{L}_{X_{\xi
}}\eta _{L}=dh_{\xi }$, and the function $h_{\xi }$ is a first integral of
the dynamical vector field $\xi _{L}$.\medskip
\end{proposition}

\noindent \textbullet\ We conclude by briefly surveying representative
contributions that explore the interplay between symmetries, conservation
laws, and dissipation in contact Hamiltonian and related geometric settings.
Gaset \textit{et al.} in \cite{GLR} analyze symmetries and conservation laws
for time-dependent contact Hamiltonian systems, establishing Noether-type
results in the presence of explicit time dependence and contact dissipation.
Rivas\textit{\ et al.} in \cite{RSS} study Lagrangian field equations in $k-$%
contact geometry as a geometric framework for dissipative field theories.
They establish Noether-type theorems relating infinitesimal symmetries to
balance laws given by conserved currents modified by dissipative terms.

Bravetti \textit{et al.} \cite{BGT} introduce conformal symmetries of the
contact Hamiltonian that generate Noether-type invariants, identifying
thermodynamic entropy as a conserved quantity along dissipative
trajectories. They extend the classical Noether theorem to Herglotz-type
equations and provide a geometric classification of symmetries into strict,
conformal, and dynamical classes according to their role in the contact
dynamics. Bravetti \textit{et al.} \cite{BJS} analyze scaling symmetries
within the framework of contact geometry and develop a natural contact
reduction procedure. They show how such symmetries can lead to stable
periodic orbits in systems with time-dependent friction, providing a
geometric realization of Poincar\'{e}'s dream, namely, the idea that
dissipative mechanisms may give rise to structurally stable periodic motion.

Zadra and Seri \cite{ZS} study symmetries of contact Hamiltonian systems.
They distinguish between strict symmetries and conformal symmetries of the
contact structure, and explain how each type affects the dynamics. They also
analyze the local Lie algebra formed by these symmetries and obtain
conditions under which the corresponding symmetry distributions are
integrable.

\section{Dirac structures on Contact Manifolds}

Dirac structures, first introduced in \cite{C,CW}, were developed as a
unifying framework generalizing both symplectic and Poisson geometries.
Defined via the Courant bracket, they encode, in a single geometric object,
both the closedness of a $2-$form and the Poisson property of a bivector
field. Originating in Dirac's work on constrained systems, they have become
an important tool in mathematical physics.

Consider the Pontryagin bundle $TM\oplus T^{\ast }M$ over a differentiable
manifold $M$. This bundle carries a natural fiberwise non-degenerate
symmetric bilinear form $\langle \cdot ,\cdot \rangle $ defined by

\begin{equation*}
\langle (X,\alpha ),(Y,\beta )\rangle =\alpha (Y)+\beta (X),
\end{equation*}%
where $X,Y\in T_{p}M$ and $\alpha ,\beta \in T_{p}^{\ast }M$ for $p\in M.$

A Dirac structure on $M$ is a subbundle $\mathbb{L}\subset TM\oplus T^{\ast
}M$ satisfying $\mathbb{L}=\mathbb{L}^{\perp }$ where $\mathbb{L}^{\perp }$
is the orthogonal complement of $\mathbb{L}$ with respect to $\langle \cdot
,\cdot \rangle .$ In particular, each fiber of $\mathbb{L}$ has dimension $%
\dim M$. To impose integrability, one considers the Courant bracket on $%
TM\oplus T^{\ast }M$, a bilinear operation generalizing the Lie bracket of
vector fields by incorporating the Lie derivative and the exterior
differential.

For sections $(X,\alpha )$ and $(Y,\beta )$ IN $\Gamma \left( TM\oplus
T^{\ast }M\right) $, it is defined by 
\begin{equation*}
\lbrack (X,\alpha ),(Y,\beta )]=\left( [X,Y],\mathcal{L}_{X}\beta -\mathcal{L%
}_{Y}\alpha -\frac{1}{2}d(\alpha (Y)-\beta (X))\right) .
\end{equation*}%
A maximally isotropic subbundle $\mathbb{L\subset }TM\oplus T^{\ast }M$ is
said to be integrable if its space of sections $\Gamma (\mathbb{L})$ is
closed under the Courant bracket, that is, 
\begin{equation*}
\lbrack \Gamma (\mathbb{L}),\Gamma (\mathbb{L})]\subset \Gamma (\mathbb{L}).
\end{equation*}%
An integrable maximally isotropic subbundle is commonly referred to simply
as a Dirac structure.

This condition includes, as special cases, the classical geometric
structures studied in Dirac geometry. Indeed, if $\mathbb{L}$ is the graph
of a 2-form $\omega $ then the closure of $\Gamma (\mathbb{L})$ under the
Courant bracket is equivalent to the condition $d\omega =0$, yielding a
presymplectic structure. Similarly, if $\mathbb{L}$ is the graph of a
bivector field $\pi $ integrability is equivalent to the vanishing of the
Schouten--Nijenhuis bracket $\left[ \pi ,\pi \right] $ so that $\pi $
defines a Poisson structure.

For a Dirac structure $\mathbb{L}$ on a manifold $M$, one can associate a
Poisson algebra of admissible functions on $M,$ \cite{Cardona2012}. This
algebra generalizes the classical Poisson algebra arising in symplectic and
Poisson geometry and provides a natural setting for the identification of
conserved quantities and first integrals in systems subject to nonholonomic
constraints or more general implicit dynamics. Moreover, it offers a
geometric framework for the formulation of implicit Hamiltonian systems, in
which algebraic constraints are intrinsically coupled with differential
equations.

In this context, a smooth function $f$ on $M$ is called $\mathbf{\mathbb{L}}%
- $admissible if there exists a smooth vector field $X_{f}$ on $M$ such that 
\begin{equation*}
(X_{f},df)\in \Gamma (\mathbb{L}).
\end{equation*}%
Any such vector field $X_{f}$ is called a Hamiltonian vector field
associated with $f$. The space of $\mathbb{L}$-admissible functions on $M$
is denoted by $C_{\mathbb{L}}^{\infty }(M)$. Although the Hamiltonian vector
field associated with a given function is not unique in general, the bracket 
\begin{equation*}
\{f,g\}=X_{f}(g)
\end{equation*}%
s well defined, satisfies the Leibniz rule and the Jacobi identity, and
endows $C_{\mathbb{L}}^{\infty }(M)$ with the structure of a Poisson
algebra. In particular, if $f$ and $g$ are $\mathbb{L}$-admissible
functions, their product $fg$ is again $\mathbb{L}$-admissible. and so is
the bracket $\{f,g\}$.

In \cite{Cardona2016}, contact structures on smooth manifolds are described
within the framework of Dirac geometry, providing one of the first
systematic approaches in this direction. Let $(M,\alpha )$ be a cooriented
contact manifold, and denote by $\mathcal{R}$, its Reeb vector field. The
tangent and cotangent bundles admit the decompositions

\begin{equation*}
TM=\ker \alpha \oplus \langle \mathcal{R}\rangle ,\quad T^{\ast }M=\langle
\alpha \rangle \oplus (\mathcal{R})_{0},
\end{equation*}%
where $(\mathcal{R})_{0}=\{\theta \in T^{\ast }(M):\theta (\mathcal{R})=0\}$
is the annihilator of $\mathcal{R}.$ These splittings naturally lead to the
following subbundles of the Pontryagin bundle $TM\oplus T^{\ast }M,$

\begin{equation*}
L_{0}=\ker \alpha \oplus \langle \alpha \rangle ,\quad L^{0}=\langle 
\mathcal{R}\rangle \oplus (\mathcal{R})_{0}.
\end{equation*}%
Both $L_{0}$ and $L^{0}$ are isotropic with respect to the natural pairing
on $TM\oplus T^{\ast }M$.

To determine whether either of these subbundles defines a Dirac structure,
one must examine their integrability with respect to the Courant bracket.
The subbundle $L_{0},$ cannot be integrable, since this would imply the
Frobenius integrability of the contact distribution $\ker \alpha $. In
contrast, $L^{0}$ is closed under the Courant bracket \cite{Cardona2016}$.$

Nevertheless, the Poisson structure induced by $L^{0}$ is trivial. Indeed,
the Hamiltonian vector fields associated with $L^{0}-$admissible functions
are necessarily multiples of the Reeb vector field $\mathcal{R},$ and
admissibility requires these functions to be invariant under the flow of $%
\mathcal{R}.$ As a consequence, for any $f,g\in C_{\mathbb{L}^{0}}^{\infty
}(M),$ the associated Poisson bracket satisfies $\{f,g\}=0.$

To avoid the triviality of the associated Poisson algebra, Cardona considers
the Dirac structure associated with the closed $2-$form $d\alpha ,$

\begin{equation*}
\mathbb{L}=\{(X,-i_{X}d\alpha )\mid X\in \mathfrak{X}(M)\}.
\end{equation*}%
Since $d\alpha $ is closed, this subbundle is maximally isotropic and closed
under the Courant bracket, hence it defines a Dirac structure on $M$.

The space of admissible functions associated with $\mathbb{L}$ consists of
those smooth functions that are invariant under the flow of the Reeb vector
field $\mathcal{R}$. The corresponding Poisson bracket is given by $%
\{f,g\}=i_{\Xi (df)}(dg)$, where $\Xi :\left( \mathcal{R}\right)
_{0}\rightarrow \ker \alpha $ denotes the inverse of the vector bundle
isomorphism $\Lambda :\ker \alpha \rightarrow \left( \mathcal{R}\right)
_{0}:X\mapsto -i_{X}(d\alpha ).$

Fundamental contributions to the theory of Hamiltonian actions and moment
maps in Dirac geometry were developed by Bursztyn and Crainic \cite%
{Bursztyn2007}, who extended the Poisson formalism by interpreting moment
maps as Dirac morphisms and clarifying the role of Lie algebra actions and
symmetries in this setting. Subsequently, the interplay between contact
moment maps and Dirac structures was analyzed by Cardona \cite{Cardona2016},
who showed how contact Hamiltonian systems can be naturally interpreted
within the Dirac framework.

As discussed in a previous section, given a smooth action of a Lie group $G$
on a contact manifold $(M,\alpha )$ that preserves the contact $1-$form $%
\alpha $ (i.e., $\mathcal{L}_{X_{\xi }}\alpha =0$ for all $\xi \in \mathfrak{%
g,}$ where $X_{\xi }$ is the fundamental vector field associated with $\xi $%
), we adopt here the functional formulation of the corresponding contact
moment map:

\begin{equation*}
\mu _{\alpha }:\mathfrak{g}\rightarrow C^{\infty }(M),\quad \mu _{\alpha
}(\xi )=i_{X_{\xi }}\alpha ,
\end{equation*}%
rather than the dual map $\mu _{\alpha }:M\rightarrow \mathfrak{g}^{\ast }.$

A Courant algebra over $\mathfrak{g}$, in the sense developed by Bursztyn
and Crainic \cite{Bursztyn-Crainic} (see also \cite{Bursztyn2007} for
related constructions in the reduction of Courant algebroids), can be
described as an exact sequence of Leibniz algebras%
\begin{equation*}
0\rightarrow \mathfrak{h}\rightarrow \mathfrak{a}\overset{\pi }{\rightarrow }%
\mathfrak{g\rightarrow }0,
\end{equation*}%
where $\mathfrak{h}$ is an abelian ideal. When $G$ is compact, Cardona \cite%
{Cardona2016}, studies the extension of the infinitesimal action $\psi :%
\mathfrak{g\rightarrow X(}M\mathfrak{)}$ to a Courant algebra morphism $\rho
:\mathfrak{a\rightarrow }\Gamma \mathfrak{(}TM\oplus T^{\ast }M\mathfrak{)}$%
, fitting into the commutative diagram

\begin{equation*}
\begin{array}{ccccccccc}
0 & \rightarrow & \mathfrak{h} & \rightarrow & \mathfrak{\alpha } & \overset{%
\pi }{\rightarrow } & \mathfrak{g} & \rightarrow & 0 \\ 
&  & {\downarrow \nu } &  & {\downarrow \rho } &  & {\downarrow \psi } &  & 
\\ 
0 & \rightarrow & \Gamma (T^{\ast }M) & \rightarrow & \Gamma (TM\oplus
T^{\ast }M) & \rightarrow & \mathfrak{X}(M) & \rightarrow & 0%
\end{array}%
\end{equation*}

A Dirac Moment Map\textbf{\ }for an extended action $\rho $ is a $\mathfrak{g%
}$-equivariant map $\mu :\mathfrak{h}\rightarrow C^{\infty }(M)$ satisfying $%
\nu =d\mu $; see \cite{Cardona2016} for details$.$ In this way, the
classical contact moment map fits naturally into the general Dirac moment
map framework developed in \cite{Bursztyn-Crainic}.

This perspective places Dirac structures within a unified geometric
framework for the reduction of nonholonomic systems via Dirac reduction,
encompassing both Lagrangian and Hamiltonian formulations, as well as
implicit (possibly degenerate) Lagrangian systems; see, for instance, \cite%
{yoshimura-marsden,jotz-ratiu}.

Since the modern formulation of mechanical systems with symmetry introduced
by Marsden and Weinstein~\cite{MarsdenWeinstein1974}, the study of
nonholonomic constraints and their associated reduction procedures has
attracted considerable attention in the literature~\cite{BS,
CushmanEtAl1995, vanDerSchaftMaschke1994}. A consistent Hamiltonian
formulation of such systems can be naturally formulated within the geometric
framework of Dirac structures~\cite{BlochCrouch1997, DalsmoVanDerSchaft1998,
vanDerSchaftMaschke1995b}. In this setting, the notion of an implicit
Hamiltonian system, characterized by a coupled system of differential and
algebraic equations, becomes fundamental. From a physical viewpoint, Dirac
structures provide an intrinsic description of constrained dynamics by
encoding the geometry of the system through the interconnection of
subsystems~\cite{BlochCrouch1997, vanDerSchaftMaschke1994}.

This viewpoint provides the fundamental motivation for introducing Dirac
structures in the present work, as they offer a natural geometric framework
for the formulation of constrained and nonconservative dynamics,
particularly in the context of contact geometry. More precisely, let $M$ be
a manifold endowed with a Dirac structure $\mathbb{L}$, and let $%
H:M\rightarrow \mathbb{R}$ be a smooth function (the Hamiltonian). The
implicit Hamiltonian system associated with the triple $(M,\mathbb{L},H)$ is
defined by the condition

\begin{equation*}
(\dot{x},dH(x))\in \mathbb{L}_{x(t)}.
\end{equation*}%
A vector field $f\in \mathfrak{X}(M)$ is said to be an infinitesimal
symmetry of the Dirac structure $\mathbb{L}$ if,

\begin{equation*}
L_{f}\Gamma (\mathbb{L})\subset \Gamma (\mathbb{L}).
\end{equation*}%
In this setting, van der Schaft \cite{vanDerSchaft1997}, studies the role of
symmetries in implicit Hamiltonian systems associated with Dirac structures.
If a vector field $f$ tis an infinitesimal symmetry of the Dirac structure $%
\mathbb{L}$ and preserves the Hamiltonian function $H$, then, under suitable
regularity assumptions, it is shown that $f$ is tangent to the constraint
manifold $M_{c}$. Moreover, the restriction $f_{c}$ of $f$ to $M_{c}$
commutes with the reduced Hamiltonian vector field $X_{H_{c}}$. This result
clarifies the compatibility between Dirac symmetries and the reduced
dynamics, and relates symmetry preservation to conservation properties of
the system.

The same work further analyzes the relationship between the existence of a
symmetry $f$ of a Dirac structure $\mathbb{L}$ and the existence of
Hamiltonian functions $F:M\rightarrow \mathbb{R}$ such that $(f,dF)\in 
\mathbb{L}.$ Within this framework, conservation results analogous to
Noether theorem arise for implicit Hamiltonian systems, where $F$ appears as
a conserved quantity for the reduced dynamics on $M_{c}$. These ideas,
developed in \cite{vanDerSchaft1997}, provide a natural motivation for
extending the analysis to dissipative settings. In particular, in contact
systems, the study of symmetries may serve as an effective tool for
identifying dissipation laws, in line with the perspective outlined in the
previous section.\medskip

\noindent \textbullet\ We conclude this section with a brief overview of
further developments on Dirac structures in contact and related geometric
frameworks. Burbulla in \cite{burbulla} extends the formalism of Dirac
structures to an extended bundle adapted to contact geometry. This
construction is based on the splitting $TM=\ker \alpha \oplus \left\langle 
\mathcal{R}\right\rangle $ induced by the Reeb vector field, together with
the subbundle $\mathrm{Ann}(\ker \alpha )\subset T^{\ast }M$ on a contact
manifold $M$. M\u{a}rcut \cite{marcut} shows that contact structures can be
realized within Dirac geometry via suitable gauge transformations of the
Courant bracket. This provides a precise way to interpret contact geometry
inside the classical Dirac framework.

Yoshimura and de Le\'{o}n \cite{yoshimura-deleon} formulate nonholonomic
mechanical systems within the framework of Dirac structures, a perspective
that naturally suggests extensions to contact and dissipative settings.
Iglesias-Ponte and Wade \cite{iglesias-ponte wade} extend the integration
theory of Dirac structures to the Jacobi and contact setting. In particular,
they prove that integrable Dirac--Jacobi structures admit a global
integration in terms of precontact groupoids. This perspective is developed
in a systematic bundle-theoretic framework by Vitagliano \cite{V}, who
introduces Dirac--Jacobi bundles as a natural extension of Dirac geometry
adapted to Jacobi and contact structures.

In the influential works of Yoshimura and Marsden \cite%
{yoshimura-marsden-b,yoshimura-marsden-c}, a Dirac formulation of implicit
Lagrangian systems is introduced, providing a geometric framework
potentially adaptable to contact Lagrangian settings. Finally, Gay-Balmaz
and Yoshimura in \cite{gay-balmaz-yoshimura} develop a Dirac-structure-based
formulation of nonequilibrium thermodynamics, providing a geometric
framework for dissipative systems that is conceptually related to contact
and Jacobi approaches.

\section{Final remarks}

Throughout this review, we have seen that contact structures and their
associated Hamiltonian dynamics have far-reaching implications, providing
powerful tools in settings where classical mechanics and symplectic
frameworks face inherent limitations. To conclude, we present two additional
topics that aim to further deepen the understanding of the intrinsic
structure of contact geometry.

A major area where the full scope of the contact setting---with its
geometric and dynamical features, including Hamiltonian vector fields,
contact moment maps, and Jacobi or Poisson algebraic structures, finds a
particularly rich expression is \textit{Geometric Quantization}. This
framework, which traces back to Dirac, aims to construct a quantum theory
from a classical one by exploiting geometric structures defined on the
classical phase space, usually a symplectic manifold.

The guiding objective is to preserve key analogies between classical and
quantum theories, such as the correspondence between classical observables
and quantum operators, as well as between classical and quantum states. The
starting point is therefore a symplectic manifold $(M,\omega )$, which
serves as the classical phase space.

We provide a very brief technical introduction following the excellent
review \cite{echeverria} and the monograp \cite{wood}.

We introduce a complex line bundle $(L,\nabla ,h)$ over $M,$ equipped with a
Hermitian connection $\nabla $, compatible with the Hermitian metric $h$,
whose curvature satisfies $\mathrm{curv}(\nabla )=-i\omega /\hbar $. The
pre-quantum Hilbert space is defined as the square-integrable sections of $%
L. $

A polarization $P$ is an integrable distribution of Lagrangian subspaces of
the complex tangent bundle $T^{\mathbb{C}}M$. Upon choosing a polarization,
the prequantum Hilbert space is restricted to polarized sections, namely
those sections annihilated by $\nabla _{X}$ for every vector field $X$ in $P$%
. In the $Spin^{c}$ formulation, the quantum space $Q(M)$ is defined as the
kernel of the Dirac operator:%
\begin{equation*}
Q(M)=\ker (D_{{}}^{+})\oplus \ker (D_{{}}^{-}).
\end{equation*}%
where $D_{{}}^{+}$ and $D_{{}}^{-}$ denote its even and odd components,
respectively.

The article \cite{Fitz-a} by Fitzpatrick extends the principles of classical
geometric quantization to the setting of contact manifolds. Let $(M,E,\theta
)$ be a compact cooriented contact manifold, where the subbundle $E=\ker
\theta \subset TM$ denotes the contact distribution. We regard the functions
in $C_{b}^{\infty }(M)$ as classical observables, where $C_{b}^{\infty }(M)$
is the subalgebra of $C^{\infty }(M)$ consisting of functions $f\in
C^{\infty }(M)$ satisfying $\xi f=0,$ where $\xi $ is the Reeb vetor field.
We assume that a compact Lie group $G$ acts on $M$ by contactomorphisms. By
averaging, the contact form $\theta $ may be chosen to be $G-$invariant.

Let $E\subset TM$ be equipped with the symplectic form $\Omega =\left.
d\theta \right\vert _{E},$ and let $\pi :L\rightarrow M$ be a hermitian line
bundle equipped with metric $h$ and connection $\nabla .$ We say that $%
(L,h,\nabla )\rightarrow (M,E,\Omega )$ is a quantum bundle if the
restriction of the curvature form of $\nabla $ to $E\otimes E$ is equal to $%
i\Omega .$

In analogy with the symplectic case, we consider the Hilbert space $H=\Gamma
_{L^{2}}(M.L).$ The polarized sections are the $CR-$holomorphic sections of
the bundle $L$, that is, those sections annihilated by $\nabla _{Z}$ for all 
$Z\in E_{1,0}.$ The map form $C_{b}^{\infty }(M)$ to the space of
skew-Hermitian operators on $H,$ given by%
\begin{equation*}
f\mapsto A_{f}=\nabla _{X_{f}}+i\pi ^{\ast }f
\end{equation*}%
is a Lie algebra homomorphism. In particular, since the components of the
contact momentum map $\Phi _{\theta }:M\rightarrow \mathfrak{g}^{\ast }$
belongs to $C_{b}^{\infty }(M),$ we can define, for each $X\in G,$ the
operator 
\begin{equation*}
A_{X}=\nabla _{X_{M}}+i\pi ^{\ast }\Phi _{\theta }^{X}.
\end{equation*}%
On the symplectic vector bundle $E,$ we choose a compatible complex
structure $J\in End(E).$ The $i-$eigenvalue $E_{1,0}\subset T_{\mathbb{C}}M$
of $J$ then defines an almost $CR-$structure on $M.$ Let $g$ be a Riemannian
metric on $M$ such that $(J,g,\theta ,\xi )$ is a contact metric structure.
This choice induces a natural connection on the Clifford bundle $Cl(E),$
compatible with its Clifford multiplication. Within this framework, the
contact Dirac operator $\mathcal{D}_{b}$ is constructed, and the quantum
space $Q(M)$ is defined as $\ker (D_{b}^{+})\oplus \ker (D_{b}^{-})$ in
direct analogy with the kernel of the Dirac operator arising in the
symplectic $Spin^{c}$ formulation of geometric quantization. This approach
is developed in \cite{Fitz-b}, where Fitzpatrick introduces the notion of a
transversally elliptic Dirac operator on contact manifolds and derives an
equivariant index formula for such operators. This index formula is then
used to define the character of the quantum representation associated with $%
Q(M)$, providing a precise description of how the Lie group $G$ acts on the
quantum space. Fitzpatrick's construction thus offers a bridge between
contact-geometric quantization and representation theory. In this framework,
symmetry reduction is reflected at the level of the equivariant index
arising in contact quantization. The index, viewed as an element of the
representation ring of the symmetry group, encodes the action of $G$ on the
quantum space, so that reduction corresponds to selecting the appropriate
invariant components, thereby determining the resulting quantum
representations.

As the final topic of this report, we briefly address a subject of
considerable recent interest that is tangential to our main theme, yet
relies on many of the same geometric ingredients. \textit{Generalized
Geometry} was introduced by Hitchin \cite{hitchin} as an extension of
differential geometry to the Pontryagin bundle $E=TM\oplus T^{\ast }M.$ Its
fundamental structure is an exact Courant algebroid on $E$, defined by the
Dorfman bracket: 
\begin{equation*}
(X_{1},\xi _{1})\circ _{H}(X_{2},\xi _{2})=\left( \left[ X_{1},X_{2}\right]
,L_{X_{1}}\xi _{2}-i_{X_{2}}\xi _{1}-i_{X_{1}}i_{X_{2}}H\right)
\end{equation*}%
where $H$ is a closed $3-$form on $M$ encoding the intrinsic geometry of $E.$
This structure \cite{wright} unifies and generalizes classical geometric
structures on $M$, including complex, metric, K\"{a}hler, and Calabi--Yau
structures, as well as Clifford actions of sections of $\Gamma (TM\oplus
T^{\ast }M)$ on spinors.

The odd-dimensional counterpart of generalised geometry, namely generalised
contact geometry, is still at an early stage of development \cite{wright,
iglesias-wade, vaisman-a, vaisman-b, poon-wade}. In this work, we begin by
reviewing the characterization of contact $1-$forms in terms of $\mathcal{E}%
^{1}(M)-$Dirac structures, as introduced by Wade \cite{wade}. In this
setting, we consider the vector bundle $\mathcal{E}^{1}(M)=(TM\times \mathbb{%
R})\oplus (T^{\ast }M\times \mathbb{R}),$ endowed with the bilinear form 
\begin{equation*}
\left\langle (X_{1},f_{1})+(\alpha _{1},g_{1}),(X_{2},f_{2})+(\alpha
_{2},g_{2})\right\rangle =\alpha _{1}(X_{2})+\alpha
_{2}(X_{1})+f_{1}g_{2}+f_{2}g_{1}
\end{equation*}%
together with an extension of the Courant bracket $\left[ .,.\right] ,$
defined in \cite{wade}. A subbundle $L$ of $\mathcal{E}^{1}(M)$ is said to
define an $\mathcal{E}^{1}(M)-$Dirac structure it it is $\left\langle
.,.\right\rangle -$maximally isotropic and integrable. A notable example is
the maximally isotropic sub-bundle $L_{{}}^{(\omega ,\eta )}$ of $\mathcal{E}%
^{1}(M)$ associated with a $2-$form $\omega $ a $1-$form $\eta $ on $M,$
defined pointwise by%
\begin{equation*}
L_{x}^{(\omega ,\eta )}=\left\{ (X,f)_{x}+\left( i_{X}\omega +f\eta
,-i_{X}\eta \right) _{x}\right\} \text{\quad\ }X\in \mathfrak{X}(M),\text{ }%
f\in C^{\infty }(M).
\end{equation*}%
Moreover, the space of sections $\Gamma (L_{{}}^{(\omega ,\eta )})$ is
closed under the extended Courant bracket $\left[ .,.\right] $ if and only
if $\omega =d\eta .$ In this case, the $\mathcal{E}^{1}(M)-$Dirac estructure
associated with a $1-$form $\eta $ will be denoted by $L^{\eta }.$ Then the
following fundamental results holds \cite{iglesias-wade} holds: the
subbundle $L^{\eta }$ corresponds to a contact $1-$form $\eta $ if and only
if 
\begin{equation*}
L^{\eta }\cap \left( (TM\times \{0\})\oplus \left( \{0\}\times \mathbb{R}%
\right) \right)
\end{equation*}%
is a $1-$dimensional sub-bundle of $\mathcal{E}^{1}(M)$ generated by and
element of the form $(\xi ,0)+(0,-1).$

Now we turn to the concept of generalized complex structure. Let $M$ be a
smooth even-dimensional manifold. A generalized almost complex structure on $%
M$ is a subbundle $E\subset \left( TM\oplus T^{\ast }M\right) \otimes 
\mathbb{C}$ which is isotropic and satisfies $\left( TM\oplus T^{\ast
}M\right) \otimes \mathbb{C}=E\oplus \overline{E},$ where $\overline{E}$
denotes the complex conjugate of $E.$ Equivalently, there is a one-to-one
correspondence between generalised almost complex structures and
endomorphisms $\mathcal{J}$ of $TM\oplus T^{\ast }M$ satisfying $\mathcal{J}%
^{2}=-id$ and which are orthogonal with respect to the natural pairing $%
\left\langle .,.\right\rangle .$ The work \cite{iglesias-wade} provides
explicit examples and classical realizations of this notion in the context
of contact and cosymplectic structures. Moreover, several distinguished
subclasses of almost contact structures, such as contact metric, Sasakian,
and $K-$contact structures, are expected to benefit from further insight
provided by Dirac-theoretic methods.

\end{document}